%% file: main.tex
\documentclass{amsart}

\usepackage{amssymb,amsmath, enumitem,graphicx,tikz,import,transparent}
\usepackage[colorlinks=false]{hyperref}
\usepackage[font=small]{caption}

\newcommand{\tts}{{\tt s}}
\newcommand{\ttc}{{\tt c}}

\newcommand{\bfS}{{\bf S}}
\newcommand{\p}{{\bf p}}
\newcommand{\q}{{\bf q}}
\newcommand{\z}{\mathbf z}
\newcommand{\f}{\mathbf f}
\newcommand{\TT}{\mathbf  T}

\renewcommand{\v}{{\bf v}}
\renewcommand{\b}{{\bf b}}
\renewcommand{\u}{\mathbf{u}}
\newcommand{\w}{\mathbf{w}}

\newcommand{\Et}{\mathrm{E}_2}
\newcommand{\SEt}{\mathrm {SE}_2}
\newcommand{\SOt}{\mathrm {SO}_3}
\newcommand{\PSLt}{\mathrm {PSL}_2(\R)}
\newcommand{\Isom}{\mathrm{Isom}}
\newcommand{\Aut}{\mathrm{Aut}}

\newcommand{\GL}{\mathrm{GL}}
\newcommand{\diag}{\mathrm{diag}}
\newcommand{\sech}{\,\mathrm{sech}}

\newcommand{\isom}{\mathfrak{isom}}
\newcommand{\set}{\mathfrak{se}_2}

\newcommand{\g}{\mathfrak{g}}

\newcommand{\C}{\mathbb C}
\newcommand{\R}{\mathbb R}
\newcommand{\Z}{\mathbb Z}

\newtheorem{theorem}{Theorem}[section]
\newtheorem{proposition}[theorem]{Proposition}
\newtheorem{corollary}[theorem]{Corollary}
\newtheorem{lemma}[theorem]{Lemma}

\theoremstyle{definition}
\newtheorem{definition}[theorem]{Definition}
\newtheorem{remark}[theorem]{Remark}

\newcommand{\on}{ orthonormal }

\newcommand{\sR}{sub-Riemannian } 

\newcommand{\be}{\begin{equation}}
\newcommand{\ee}{\end{equation}}

\newcommand{\sn}{\smallskip\noindent}
\newcommand{\mn}{\medskip\noindent}

\newcommand{\sm}{\sf\small}

\title{Bicycle paths, elasticae  and  sub-Riemannian geometry}

\author[A. Ardentov]{Andrey  Ardentov}
\address[Ardentov]{Ailamazyan Program Systems Institute, Russian Academy of Sciences, Pereslavl-Zalessky, Russia} 
\email{aaa@pereslavl.ru}
\author[G. Bor] {Gil Bor} 
\address[Bor]{CIMAT, A.P. 402, Guanajuato, Gto. 36000, Mexico} 
\email{gil@cimat.mx}

\author[E. Le Donne]{Enrico Le Donne} 
\address[Le Donne]{Dipartimento di Matematica, Universit\`a di Pisa, Largo B. Pontecorvo 5, 56127 Pisa, Italy \\
\& \\
University of Jyv\"askyl\"a, Department of Mathematics and Statistics, P.O. Box (MaD), FI-40014, Finland}
\email{ledonne@msri.org}

\author[R. Montgomery]{Richard Montgomery}
\address[Montgomery]{Mathematics Department\\ University of California, Santa Cruz\\
Santa Cruz CA 95064}
\email{rmont@ucsc.edu} 

\author[Yu. Sachkov]{Yuri Sachkov}
\address[Sachkov]{Ailamazyan Program Systems Institute, Russian Academy of Sciences, Pereslavl-Zalessky, Russia} 
\email{yusachkov@gmail.com}

\date{\today}

\begin{document}

\maketitle


\begin{abstract} 
 We relate  the   \sR geometry  on the group of rigid motions of the plane to  `bicycling mathematics'. We show that this geometry's   geodesics  correspond to bike paths whose  front tracks   are either  non-inflectional Euler elasticae or straight lines, and that   its infinite minimizing geodesics (or `metric lines')
  correspond to bike paths   whose front tracks are either   straight lines or `Euler's solitons'  (also known as  Syntractrix or Convicts' curves).  
 \end{abstract}

\tableofcontents

\section{Introduction}
An oriented line segment  of fixed length $\ell$ moves  in the Euclidean plane.  We think of the segment as a  bicycle so that  its  end points mark the points of 
contact of the front and back wheels with the ground.  As the segment moves, its end points trace a pair of curves,   the {\em front }and  {\em  back tracks}. 
  We impose  the  `no-skid' condition on the motion: the line segment must  be tangent to the back track at each instant.
  Any such  motion of a line segment  will be  called a {\em bicycle path}. See Figure~\ref{fig:bicycling}. We define the  {\em length} of a bicycle  path    to be the ordinary Euclidean  length of its front track. 


\mn\begin{figure}[ht]
\centering\def\svgwidth{\textwidth}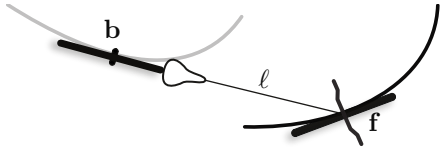
\caption{ The front and back  tracks of  a bicycle path (the dark  and light curves, respectively).}
\label{fig:bicycling}

\end{figure}

 {\em  What are the minimizing bike paths?} These are bike paths  whose length  minimizes the   length among all  competing bike paths which  connect 
 two given placements  of the line segment.  
 
 We will say that two curves in the plane {\it have the same shape} if one curve can be taken onto the other by
 a homothety, that is, a composition of an isometry and a dilation.   
 The {\it width} of a plane curve is the infimum of the distances between two parallel lines
 which bound a strip containing that curve. 

\begin{theorem}\label{thm:main1} The front track of a  minimizing bicycle path  is a straight line or an arc of a {\em non-inflectional  elastic curve} of  width twice the bicycle length  or less.  Every possible  shape of non-inflectional elastic curve arises in this way.\end{theorem}

\begin{figure}[ht]
\centering
\centering\def\svgwidth{\textwidth}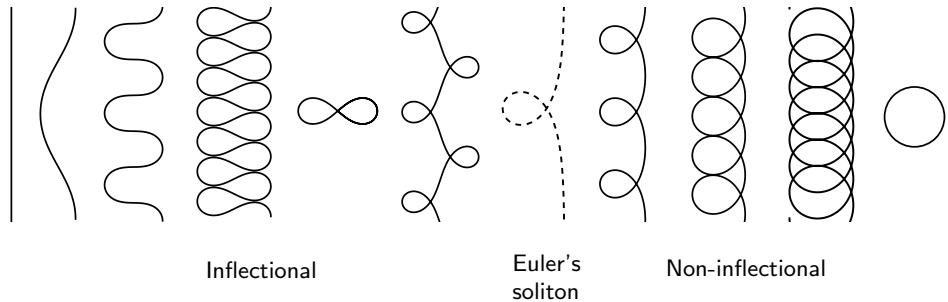
\caption{ The family of elastic curves. }
\label{fig:elastica-intro}
\end{figure}


See  Figure \ref{fig:elastica-intro} for some examples of elastic curves, also known as {\em elasticae}, a remarkable family of plane curves  studied by Jacques Bernoulli (1691), Euler (1744) and many others.  (We recommend  \cite{Levien} for a  nice historical review.)   Elasticae can be parameterized by  elliptic functions.   They  are the planar curves having  critical total curvature squared, among all  curves with fixed length connecting two given points.  They are  defined by the differential equation \eqref{eq:energy_form}  below. Another 
characterization  of elasticae is  as  curves whose curvature varies linearly with the (signed) distance to some fixed line, the {\it directrix} of the elastica. (Can you see this line for each of the curves in Figure \ref{fig:elastica-intro}?). 
Theorem \ref{thm:main1}  provides yet another characterization of elasticae, apparently new. In  Figure \ref{fig:elastica-intro}  the Euler soliton and all  the curves to its right  are `non-inflectional': they have no  points with null curvature.  
All the elasticae to the left of  the Euler soliton are inflectional.  See Section \ref{sec:elasticae} below for more information on elasticae. 

\sn

In Section \ref{sec:geod} we derive relationships  between the shapes and widths   of the   elasticae of Theorem \ref{thm:main1}.
In general, a bike path  {\em is  not determined} by its front track. That is, for a given front track, there is a circle's worth of corresponding back tracks, each of which determined by the bicycle frame orientation at some fixed  point of the front track. 
However, for each of the minimizing  bike paths of Theorem \ref{thm:main1},  except those whose front track is a line segment, its  front track,  combined with the condition that
the bike path  minimizes,   {\em does determine} the back track. 
For a given  shape of a non-inflectional elastica  there are two distinct    types of minimizing  bike paths: one whose front track has  width $2 \ell$
and another of certain lesser width (depending on the shape).  
 We call them `wide' and `narrow' front paths.  (Exception:  Euler solitons appear only in width $2\ell$.)  The shapes of the back tracks  of these two types  are  quite different. See Figure \ref{fig:geod} and Proposition 
\ref{prop:width} for the full details. 

\sn 

Let us emphasize that Theorem \ref{thm:main1} does {\em not} state that arbitrary subsegments of a given non-inflectional elastica occur as front tracks of minimizing bike paths. In fact, typically, the opposite is true. Consider for example Figure \ref{fig:conj}. It depicts a geodesic bike path connecting two horizontal placements of the bike. Clearly, this is not a minimizing path; a straightforward eastward ride will be much shorter. Theorem \ref{thm:main1} only states  that  short enough subsegments of this path  are  minimizing between their endpoints.
We do not address here how short is `short enough'.  
For comprehensive results in that direction
see \cite{Sachkov1,Sachkov2,Sachkov3}.   
According to our next theorem, the fact that geodesics eventually fail to minimize, as   depicted in Figure \ref{fig:conj}, is  typical, with two  exceptions. 

\begin{figure}[ht]
\centering
\includegraphics[width=.6\linewidth]{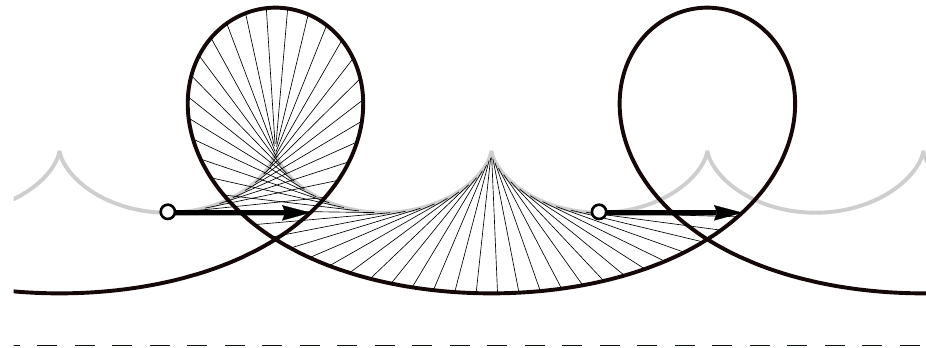} 
\caption{ A non-minimizing geodesic segment.}
\label{fig:conj}
\end{figure}

\begin{theorem}\label{thm:main2}
An infinitely long bike path is a global minimizer, that is, all of its compact  subsegments 
minimize length between their end points,  if and only if it is  one of the following two types:
\begin{enumerate}
\item its front track is a straight line and  its
back track is a  tractrix or a straight line, or
\item its front track is an  Euler soliton
 of width twice the bike length and its back track is a  tractrix. 
  \end{enumerate}
See Figure \ref{fig:kink}.
Furthermore, there is an  isometric involution of   the  bicycle
configuration space which takes  paths of one type to paths of the other,  
provided the back track of the path is a tractrix  and not a line.  See Lemma \ref{flip}.
\end{theorem}

In the soliton case, at the `highest point' of the soliton curve, that is, at  its  point of maximum curvature,
the bike frame is  oriented perpendicular to the directrix,  pointing away from it.   
 For an explicit parametrization of the  soliton and   tractrix,  see Lemma  \ref{lemma:tractrix} below.  

\begin{figure}[ht]
\centering
\includegraphics[width=.8\linewidth]{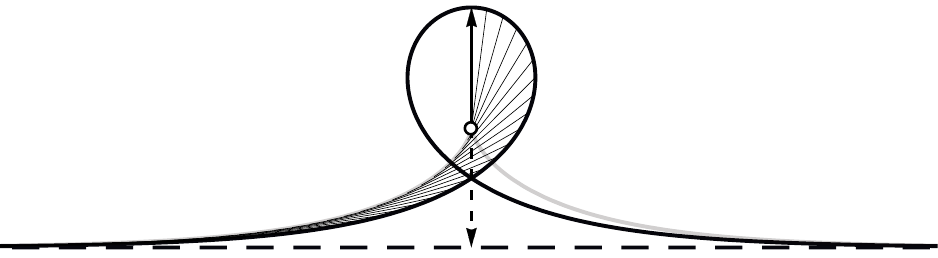} 
\caption{ Two infinite minimizing bike paths share the tractrix (light curve) as a common back track; the two front tracks are a straight line (dashed dark horizontal line) and  an `Euler's  soliton' (solid dark curve).}
\label{fig:kink}
\end{figure}

\sn{\bf About the proofs.} With one notable exception, the proofs of the  two theorems above, once set up in the appropriate language,   reduce  to standard   calculations with the geodesic equations of    \sR geometry. Such a calculation yields Theorem \ref{thm:main1} and `one half' of Theorem \ref{thm:main2}; namely, that all geodesics, except the two types mentioned in Theorem \ref{thm:main2}, are not globally minimizing (the argument  for the last statement is essentially contained in  Figure \ref{fig:conj}). That  bike paths whose front track is  a straight line are global minimizers 
follows directly  from the  definition of bike path length.  What  remains   to show is that geodesics of Theorem \ref{thm:main1} whose front tracks are Euler solitons are global minimizers. Here,
the notable exception mentioned above, we found a surprisingly simple proof, inspired by `bicycle mathematics'. The so called `bicycle transformation' (or  Darboux transformation or B\"acklund transformation
or  flip) consists of rotating a bicycle by $180^0$ about its rear end. It is easy to check that this transformation is an  isometric involution on the bicycle configuration space, so takes global minimizers to global minimizers. Applying it to a (generic) global minimizer whose front track is a line,  we obtain a global minimizer whose front track is an Euler soliton, as depicted in Figure 
\ref{fig:kink}. 

\sn{\bf Comparison with Previous Works.}   
One  of us has published a series of  works  \cite{Sachkov1, Sachkov2, Sachkov3} on the   
geodesics and their minimality (or `optimal synthesis' in the language of control theory) for this same subRiemannian geometry. These earlier works  focused only on the back wheel projection. The front wheel  was not present.    What is new in  our  work is the focus on the front wheel projection and  the realization that 
the front wheel  traces out    elasticae.    We could have derived  our   minimality results by translating the earlier    results  from the back wheel over to the  front wheel
but we have found it simpler and more illuminating to   
directly  study the geodesics from the front wheel point of view.   

Our  other new contributions    are the  subRiemannian involution taking straight line tracks to Euler solitons  (Lemma \ref{flip}, Theorem \ref{thm:isometries}) 
and the relations sketched in subsection \ref{heuristics} 
between the   geodesics here  and those occurring when rolling the  hyperbolic plane along the Euclidean plane 
as investigated by  Jurdjevic \cite{Jurdjevic1, Jurdjevic2}.  

\sn{\bf Computer graphics and animations.} Most  figures  in this article were made using the computer program Mathematica. They are complemented with some  `bicycle mathematics' animations, found on the web page \url{https://www.cimat.mx/~gil/bicycling/}. 

\sn{\bf Acknowledgments.}
AA and YS  were supported by the Russian Science Foundation under grant 17-11-01387-P and performed in Ailamazyan Program Systems Institute of Russian Academy of Sciences.
E.L.D. was partially supported by the Academy of Finland (grant288501 `\emph{Geometry of \sR  groups}' and by grant 322898 `\emph{Sub-Riemannian Geometry via Metric-geometry and Lie-group Theory}') and by the European Research Council (ERC Starting Grant 713998 GeoMeG `\emph{Geometry of Metric Groups}'). GB was supported by CONACYT Grant A1-S-4588.  

\section{Wider Context} 

For a number of surprising theorems around bike paths, and their relations to integrable systems, see
\cite{biking}.  

The bicycling configuration  space  is diffeomorphic to the three-dimensional Lie group $\SEt$ of rigid motions 
of the plane (orientation preserving isometries).  Its length structure comes from a left-invariant \sR metric on this group.  See Section \ref{sec: Config} below for details.  Such a structure is  unique up to scale, \cite{Agrachev-Barilari3D, Bor2},
and that scale can be interpreted as the length of the bicycle frame.  This  structure,   from the perspective of the {\it back wheel}
track, has been investigated by many authors \cite{Citti, Sachkov1, Sachkov2, Sachkov3, Hladky}  and  used  to understand   aspects of mammalian vision. In that latter context
the group $\SEt$ is typically referred to as  the
``roto-translational group'' and the orientation of the bicycle frame is the crucial object, as optical processing in the brain involves cells whose function is to
perceive orientations of line segments.  

Gershkovich and Vershik gave a general description  and classification of left invariant \sR structures on three-dimensional Lie groups in \cite{Gershkovich}, see also  \cite{Agrachev-Barilari3D}.
In all cases the   geodesic equations are those of `generalized elastica'.    

On any metric space we can speak of  `globally minimizing geodesics' or, synonymously, `metric lines': isometric embeddings of the real line into the metric space.  See \cite{Burago}.
 What are the metric lines for a given  \sR structure?   
Theorem \ref{thm:main2} 
 answers this question  for the bicycling case.  
 
 Hakavuouri and LeDonne \cite{LeDonne} 
prove a number of powerful general theorems  regarding  metric lines in \sR geometries by implementing the   operation of
``blowing down'' a geodesic.  Sufficient iterations of  blow-down yield a line in a Euclidean space.
As a corollary, they prove that if a \sR geometry $Q$ comes, like ours, with a \sR submersion $\pi$ to the Euclidean plane $\R^2$,
then (1) the projection of any metric line  in $Q$ must  lie a bounded distance
from a line in the plane, and (2) if that planar  line    is given by $x= 0$ and if we write the projected geodesic  as $(x(t), y(t))$
then $x(t)$ cannot be a non-constant  periodic function.   Item (2)
excludes all the   elasticae of Theorem  \ref{thm:main2}   besides the line and the soliton from being   metric lines.

We know five  other  rank 2 \sR geometries besides our $\SEt$ geometry whose geodesics
project to   elasticae under a \sR 
submersion onto  the Euclidean plane.    (See the third paragraph of Section \ref{sec: Config}  for the definition of a `\sR submersion'.) 
Two   are Carnot geometries, one being the Engel group, whose   growth vector  is $(2,3,4)$ (see \cite{Ardentov1, Ardentov2}),  and
the other, sometimes called  the Cartan group,  being the unique   Carnot group  with growth vector $(2,3,5)$ (see \cite{Sachkov4}).  
(The growth vector of a Carnot group is its basic numerical invariant and 
encapsulates the graded dimensions of its Lie algebra.)
Another  is the  flat Martinet geometry, see \cite{martinet}.
The remaining  two  are five-dimensional, 
arising from  rolling  a   constant curvature
surface along the Euclidean plane, and have   state spaces   $\SOt \times \R^2$ and $\PSLt\times \R^2$. 
See \cite{Jurdjevic1, Jurdjevic2} for a derivation of elasticae as their geodesics. 
  In all five  geometries the geodesics  projecting to 
 Euclidean lines are metric lines and in the   flat Martinet geometry these exhaust the set of  metric lines. 
 In the other four geometries some of  the geodesics which project onto solitons are also metric lines.
 In the  $\SOt \times \R^2$ case  all elasticae, and in particular all Euler solitons,  arise 
 as  projections of geodesics  onto $\R^2$,  but only some of the solitons, namely those whose kink is `small enough',
  arise as projections of  metric lines.

 {\em  Are all these occurrences of elasticae and Euler solitons in \sR geometries  related?}  There is a  \sR submersion from the  Cartan  group onto the  Engel group,  so that the space of   Engel  geodesics embed into the space of  Cartan geodesics by  horizontal lift. Similarly, the space of flat Martinet geodesics embed into the space of Engel geodesics.
   The   bike configuration 
	space $Q = \SEt$
 can be constructed as a circle bundle associated   to  the hyperbolic rolling space $\PSLt \times \R^2$, viewed as
 a principal $\PSLt$ bundle,  and this fact and its related geometry allows us to embed the bike geodesics into the
 hyperbolic rolling geodesics.     See the last paragraph of Section \ref{subsec:bike transport} below.  We leave the 
 possibility  of uncovering   relations between the other pairs of geometries and of
 some deeper reason underlying  the ubiquity of   elasticae in  \sR geometry  to future researchers.

\section{Concepts building to the  proofs}  

\subsection{Elasticae}\label{sec:elasticae}
An immersed   plane curve   is an {\it elastica} if its  curvature $\kappa (t)$,  as a function of arc length $t$,
satisfies the 2nd order  ODE
$$\ddot \kappa+\frac{1}{2} \kappa ^3 + A \kappa =0$$
for some constant $A$.  See, for example, \cite{elnotes}. 
This  is an equation  of Newton's type,  with potential 
$\frac{1}{8} \kappa^4 + \frac{1}{2}A \kappa ^2$. 
Consequently,   there is a constant  `energy' $B \in \R$ such that
\begin{equation} 
\label{eq:energy_form} 
\frac{1}{2} (\dot \kappa) ^2 + \frac{1}{8} \kappa^4 + \frac{A}{2} \kappa ^2 = B.
\end{equation}
We call the latter equation  the `energy form' of the elastica equation.
If $\kappa(t_0) = 0$ at some point then the energy equation asserts that $B \ge 0$.
Consequently, if $B < 0$ we must have that $\kappa$ never vanishes along the curve.
Since $\kappa(t_0) = 0$ corresponds to an inflection point of the curve, we call such
elasticae `non-inflectional.'   
Elasticae for which $B =0, A< 0$ are also non-inflectional and consist of the  Euler solitons. All non-inflectional elasticae, except the Euler solitons,  have periodic curvature.

Equation \eqref{eq:energy_form} can be rewritten (by `completing the square') as 
\be
\label{eq:sq}
\dot\kappa^2+\left({\kappa^2\over 2}+A\right)^2=2B+A^2.
\ee
Thus the parameters must satisfy $2B+A^2\geq 0.$  The set of elasticae is invariant under dilations.  To 
 dilate an immersed plane curve $c(t)$ parameterized by arclength $t$ by a factor $\lambda > 0$  we form  $\tilde c (t) = \lambda c( t /\lambda)$.
 The dilated curve $\tilde c(t)$  is still parameterized by
 arclength and has curvature  $\tilde \kappa (t)  = \frac{1}{\lambda} \kappa (\frac{t}{\lambda})$. It follows
 by a direct computation that the $\lambda$-dilate of an elastica satisfying equation \eqref{eq:energy_form} with  parameters $A,B$
   satisfies a new equation \eqref{eq:energy_form},  now  with rescaled parameters $\tilde A = A/\lambda^2, \tilde B = B/\lambda^4$.  
 Thus 
 \be\label{eq:mu}
 \mu:=-2B/A^2\leq 1
 \ee 
 is scale invariant and can be thought of as a `shape parameter'. Non-inflectional elasticae correspond to $B\leq  0$, that is, $0\leq \mu\leq 1$, in which case  $A\leq  0$ as well. There are 3 types of `exceptional' elasticae, lines, circles and solitons. 
Lines correspond to  solutions of equation \eqref{eq:energy_form} with  $\kappa=B=0$,  solitons  to  $B=0, A<0, \kappa\neq 0$, and  circles to  $\mu=1,$ that is, $B=-A^2, $ $\dot\kappa=0.$ See Figure \ref{fig:param}. 


\begin{figure}[ht]
\centering\def\svgwidth{\textwidth}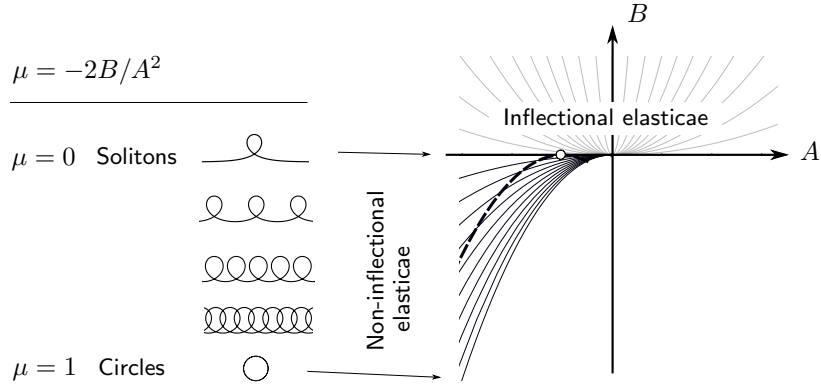
\caption{ Elasticae parameter space. The light   and dark solid curves parametrize inflectional ($B>0$) and non-inflectional  ($B<0$) elasticae (respectively) of  `constant shape', level curves of the shape parameter $\mu=-2B/A^2$ of equation \eqref{eq:mu}.  The dashed heavy  curve in the third quadrant corresponds to the elasticae which  appear  as front tracks of geodesic bike paths for fixed
bike length $\ell=1$ (see Proposition \ref{prop:kappaeq}, where this curve is parametrized by $a$). Its intersection with the $A$-axis (marked with a white dot) stands for the Euler soliton ($a=1$ in equation \eqref{eq:AB}). Each non-zero level curve of $\mu$ in the third quadrant intersects the dashed curve at 2 points, corresponding to the two sizes of non-inflexional elasiticae appearing as front tracks of bicycle geodesics, `wide' and 'narrow' (to the left and right of the white dot, respectively). }
\label{fig:param}
\end{figure}

\subsection{Configuration  space. Metric concepts.}

\label{sec: Config} 

We begin by  reformulating our  theorems in the language of  \sR and metric geometry. 

Let  $\ell>0$ be the bicycle length. Then the configuration space for bike motions  can be expressed as   $Q=\{(\b,\f)\in\R^2 \times \R^2\,|\, \|\f-\b\|=\ell\} \subset \R^2 \times \R^2$, where $\|\cdot\|$ is the standard Euclidean norm on $\R^2$. It is easy to see that $Q$ is  a smooth manifold diffeomorphic to $\R^2 \times S^1$.
 The no-skid condition defines  a rank 2 distribution $D\subset TQ$ on $Q$ by saying that 
a vector $(\dot\b,\dot\f)\in T_{(\b,\f)}Q$ belongs to $D_{(\b,\f)}$ if and only if $\dot\b$ is a multiple of  $\f-\b$. {\it Bike  paths are    the  integral curves of} $D$.  

$D$ is a  {\em contact distribution}. We prove this in    Lemma \ref{lemma:contact} below.
Alternatively,   in Section \ref{sec:models}   we show how to identify 
$Q$  with the  space of (oriented) tangent lines to the plane, also known as ``contact elements'' since they represent 1st order contact
of curves.  In this context   $D$ is  the canonical contact distribution on this space of contact elements, 
  one of the  first   examples
of a contact manifold.  See for example Appendix 4 of  Arnol'd's famous   book,  \cite{Arnold}.

  Let $\pi_f:Q\to\R^2$ be the front  wheel projection, $(\b,\f)\mapsto\f$.  $D$ is transverse to the fibers of $\pi_f$, hence one can equip $D$ with an inner product   by pulling back the Euclidean metric on $\R^2$ to $Q$ by  $\pi_f$, then restricting  to $D$. The 3-manifold $Q$, together with the distribution $D$ and the  inner product on it, is an example of  a {\em \sR manifold.}  
We   constructed the inner product on $D$ in such a way that  the front wheel projection is a {\it \sR submersion}: 
  for each $\q \in Q$ the differential $d \pi_f (\q)$ maps the  2-plane $D_\q$  isometrically onto $T_{\pi_f (\q)} \R^2 = \R^2$.
	This \sR structure is isometric to the standard \sR structure on the group $\SEt$ studied in  \cite{Citti,  Hladky, Sachkov1, Sachkov2, Sachkov3}. Up to scaling and isometries, it is the unique left-invariant \sR structure on $\SEt$ of contact type.

  Since $Q$ is connected and $D$ is  contact,  the  Chow-Rashevskii  Theorem \cite{tour}   implies that any two points in $Q$ are connected  by  a bicycle path. The length of such a path is defined using the inner product on $D$, just as in Riemannian geometry. 
   In view   of our construction  of the inner product, the length of a bike path equals the   length of its front wheel projection to $\R^2$, as asserted in the introduction. 

%
Defining the distance between two points of $Q$ to be  the infimum of the  lengths of the    bike  paths connecting them turns $Q$ into a metric space. A
 {\em minimizing geodesic} in $Q$ is a bike  path $\gamma:I\to Q$, where $I\subset \R$ is a compact interval, realizing  the  distance between its end points. A {\em geodesic} is  a   bike  path $\gamma: I \to Q$, where $I\subset \R$ is an  interval (possibly non-compact), such that  every   $t_0\in I$  is  contained  in a compact subinterval $I'\subset I$ for which  $\left.\gamma\right|_{I'}$ is a   minimizing geodesic. In addition, we require that if 
 $t_0$ is an interior point of $I$ then $t_0$ is also   an interior point of $I'$. (This last condition excludes arbitrary concatenations of 
 minimizing geodesics from being geodesics.)

Theorem \ref{thm:main1} states that the $\pi_f$-image  of any geodesic  is a non-inflectional  elastic curve or a straight line.  
A {\em metric line} in $Q$ is an infinite geodesic  all of whose compact subsegments are minimizing geodesics. 
Theorem \ref{thm:main2} states that the $\pi_f$-image  of any metric line   is either a Euclidean line or an Euler soliton. (The `width' of this
soliton is twice the length of the bike frame.)  

A {\em \sR isometry} of $Q$  is a diffeomorphism that preserves $D$ and the inner product on it. 

\begin{remark}
Clearly, a \sR isometry is  a distance preserving homeomorphism. 
The latter can be taken as a weaker `metric' definition of isometry. For a general \sR manifold, the equivalence of the two definitions is  an open problem. For an {\em equi-regular} \sR structure, such as our case (or any homogeneous \sR manifold), the two notions are equivalent \cite{Capogna_LeDonne, LeDonne_Ottazzi}.
\end{remark}

By construction, the action of the group $\Et$ of isometries of the plane $\R^2$ lifts to  an action on $Q$   by \sR isometries. An element  $g \in \Et$
acts on $Q$ sending $({\bf b} , {\bf f})$ to $(g{\bf b} , g{\bf f})$ so that our   \sR submersion  $\pi_f$ intertwines the $\Et$-action  on $Q$  
with the standard action of $\Et$ on $\R^2$.    But these are not all the \sR isometries of $Q$. There is one extra symmetry that plays an important role in our proof of Theorem \ref{thm:main2}.

\begin{lemma} \label{flip}  The map $\Phi:Q\to Q$, $ ({\bf b}, {\bf f})\mapsto ({\bf b}, 2{\bf b}  - {\bf f})$, which  `flips' the bike frame about the back wheel 
is a  \sR isometry of $Q$. See Figure \ref{fig:Phi}. 
\end{lemma}

\begin{figure}[ht]
\centering
\centering\def\svgwidth{\textwidth}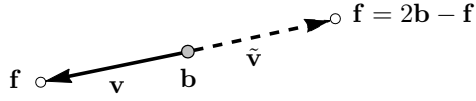
\caption{ Lemma \ref{flip}: Flipping a bike about its back wheel.}
\label{fig:Phi}
\end{figure}

\begin{proof} $\Phi$ is the restriction of a linear map to $Q\subset\R^2\times\R^2$. Thus its derivative is given by the same formula,  $(\dot\b,\dot\f)\mapsto(\dot\b,2\dot\b-\dot\f).$ It clearly preserves the no-skid condition hence it leaves $D$ invariant. It remains to show that $\|\dot\f\|=\|2\dot\b-\dot\f\|.$ 
Now decompose orthogonally $\dot\f=\dot\f_\|+\dot\f_\perp$, $\dot\b=\dot\b_\|+\dot\b_\perp$, where $\dot\f_\|, \dot\b_\|,$ are the orthogonal projections  along $\b-\f.$ The bicycling no-skid condition implies  $\dot\b_\perp=0$ and   $\|\b-\f\|=\mathrm{const}$ implies  $\dot\f_\|=\dot\b_\|$,  hence $\dot\f_\|=\dot\b.$ Thus $2\dot\b-\dot\f=2\dot\f_\|-(\dot\f_\|+\dot\f_\perp)=\dot\f_\|-\dot\f_\perp.$ That is, $2\dot\b-\dot\f$ is the reflection of $\dot\f$ about $\b-\f$. It follows that $\|2\dot\b-\dot\f\|=\|\dot\f\|.$
\end{proof}

For completeness we 
  describe the full group of isometries of  $Q$.

\begin{theorem}  
\label{thm:isometries}The   group $\Isom(Q)$ of all  \sR isometries of $Q$ is an  extension of $\Et$ by the two-element group $\Z/2\mathbb Z$.   This two-element group is 
generated   by the isometric involution  $\Phi$ which `flips the bike frame',   as described in  Lemma \ref{flip} above.  
Thus $$\Isom(Q)  \simeq \Et\rtimes \Z/2\Z \simeq  \SEt\rtimes (\Z/2\Z\times \Z/2\Z) .$$
The identity component of $\Isom(Q)$ is $\SEt$, acting freely and transitively on  $Q$ and  so 
induces  a \sR isometry between   $Q$ and a left-invariant \sR metric on $\SEt$.
\end{theorem} 

We prove this  theorem  in the appendix. 
 %
 %
Hladky \cite{Hladky2},  in his final section,  computes that  the Lie algebra of $\Isom(Q)$   is that of  $\SEt$.   The same conclusion can be  drawn from the   asphericity  of the associated CR structure, as in \cite[\S7]{Bor2}. But calculating the Lie algebra of  $\Isom(Q)$  only describes the identity component of  $\Isom(Q)$, missing the  `discrete part' 
 (or `isotropy representation'') of the isometry group, as we do in the appendix.

\section{The proof of Theorem  \ref{thm:main1} (and some more)}
\label{sec:geod}
We prove a more detailed version of  Theorem \ref{thm:main1}, subdividing the assertions into 4  claims. 
Most of these claims do not hold for  the `exceptional' elasticae (line, circle, soliton). We   first describe the non-exceptional  situation, then correct  for the  exceptional elasticae. 
 
\mn{\bf Claim 1}  (Theorem \ref{thm:main1}). The front track of each bicycle geodesic is a NIE (non inflectional elastica) or a straight line. 
  
\mn{\bf Claim 2} (Wide and narrow). For a  fixed bicycle frame of length $\ell$, each  shape of NIE appears as a front track in two different sizes, `wide' and `narrow': The wide front tracks are NIE of  width  $2\ell$. 
 A narrow front track  can have any width in  $(0,2\ell)$, depending on its shape: the more circular is  a narrow front track, the narrower it is. See Figure \ref{fig:geod}. 

 Exceptions:  the circle and the soliton appear as front tracks only with width $2\ell$. 
 \begin{figure}[ht]
\centering
\includegraphics[width=\linewidth]{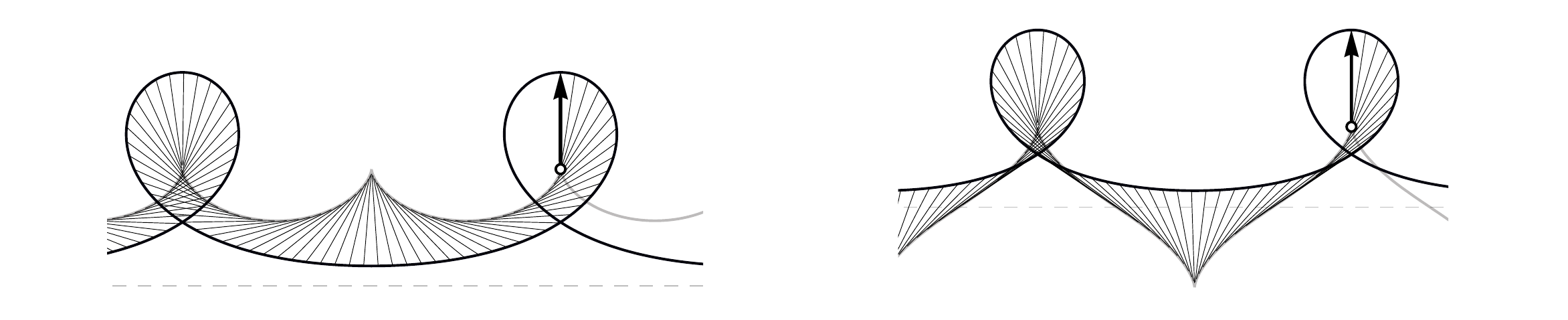} 
\caption{ Two bicycling geodesics with front tracks (dark solid curve) which are non-inflectional elasticae of the same shape but of different size: `wide' (left) and `narrow' (right).  In each of the two figures: the horizontal dashed  light line is the directrix, the light solid curve is the back track,  the  arrow depicts the bicycle frame, pointing towards the front wheel, at the moment of going through a point of maximum curvature of the front track. See Proposition \ref{prop:frame}.}
\label{fig:geod}
\end{figure}

\mn{\bf Claim 3} (Unique horizontal lift). Each non-linear NIE front track, wide or narrow, has a unique horizontal lift to a bicycle geodesic. 
%
%
This lift is determined by the
back track  found by the  following rule: at  points of maximum curvature of the front track the bicycle frame is perpendicular to the front track,  pointing `outside' the front track (that is, in the direction opposite to the acceleration vector of the front track).
 See Figure \ref{fig:geod} and our web animations \cite{anim}.

The bicycle frame is also perpendicular to the front track at the points of minimum curvature. For the wide NIE, the frame at this point also points outside the front track.
For the narrow NIE the frame  points inside.

Exception: all  horizontal lifts of a Euclidean line are  globally minimizing bike paths. Two of the lifts correspond to riding along the line, either forward or backwards,  with the bike frame aligned with the line.  The rest of the lifts correspond to the back wheel tracing a {\em tractrix} of width $\ell$ (the light solid curve of Figure \ref{fig:kink}). 

\mn{\bf Claim 4} (Flipping a front track). There is a sub-Riemannian  involution $\Phi:Q\to Q$ on the bicycling configuration space, rotating
the bicycle frame by $180^0$ about its rear end. It acts on the space of bicycle geodesics, as well as their front tracks, preserving the `narrow' and `wide' subclasses. Each NIE has its `length' $L$: the distance between two successive points along the curve of maximum (or minimum)   curvature, see Figure \ref{fig:shortcut}. The flip of a  wide NIE is obtained by translating it by $L/2$ along its directrix. The flip of a narrow NIE is obtained by a `glide reflection': translation by $L/2$ along the directrix followed by a reflection about it. See Figure~\ref{fig:nie}. 

Exceptions: the flip of the circle is the circle itself, the flip of the line is the soliton, of width $2\ell$, and vice versa.

\begin{figure}
\centering
\includegraphics[width=\linewidth]{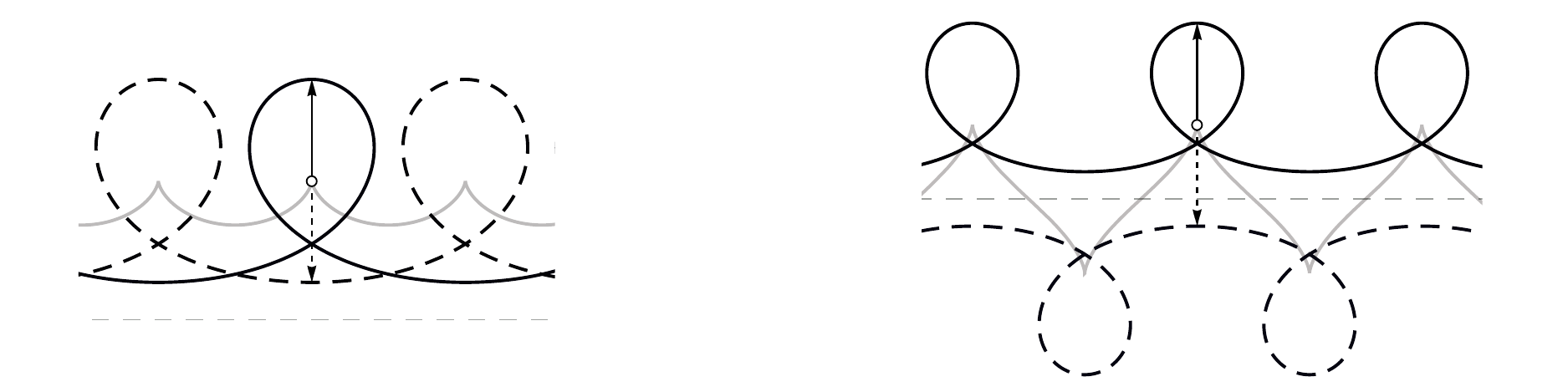} 
\caption{ The flips of bicycle geodesics with `wide' (left) and `narrow' (right)  front tracks of the same shape. In each of the two figures: the  dashed dark curve is the `flip' of the solid dark curve and vice versa, the light  solid curve is their common  back track, the light dashed horizontal  line is their common directrix, the solid  arrow indicates the bicycle frame at a point of maximum curvature of the solid front track and the dashed  arrow indicates the bicycle frame  at a point of minimum curvature of  the dashed  front track. See Proposition \ref{prop:flip}.}\label{fig:nie}
\end{figure}

\subsection{Generalities on geodesics in \sR geometry}
To prove the above 4 claims we  review some general facts from \sR geometry.  For more  details see Chapter 1 of   \cite{tour}.

 Let $M$ be a smooth manifold.  We can turn a smooth vector field $X$ on $M$   into a fiber-linear function $P_X:  T^* M \to \R$ by the
rule $P_X (\q, \p) = \p (X(\q))$, where $\q\in M$ and $\p \in T^* _\q M$.
Consider a  general rank $r$ distribution $D\subset TM$, equipped with a  \sR metric on $D$ and an  \on frame $X_1, \ldots, X_r\in\Gamma(D)$.
Form the corresponding fiber-linear functions $P_i := P_{X_i}$.  Then  the {\em normal geodesics}
of this \sR structure are, by definition,  the projections onto $M$ of the solutions to the standard Hamiltonian equation on $T^*M$, 
\begin{equation}\label{eq:normal}\dot\q=\partial_\p H,\ \dot\p=-\partial_\p H, \ \mbox{ where } H = \frac{1}{2} \sum_i (P_i) ^2 . 
\end{equation}
See Theorem 1.14 on page 9 of \cite{tour} for the full statement and later, a proof.

Normal geodesics  parametrized by arc length correspond to solutions of equation \eqref{eq:normal} with energy $H =1/2$. Short enough segments of normal geodesics are length minimizers, but the converse is not true, in general, due to the existence of {\em singular} (or {\em abnormal}) geodesics. See  \cite{tour}, particularly Chapters 3 and 5.  However, a basic result of the theory is:  {\em if $D$ is a {\em contact} distribution then all length minimizing  $D$-horizontal curves are  normal geodesics}.  See the example at the top of page 59 in \cite{tour}. 

\subsection{The bicycling geodesic equations} \label{sec:bge}

Let $Q=\{(\b, \f)\in \R^2\times\R^2\,|\, \|\b-\f\|=1\}$ be the bicycling configuration space, equipped with  the coordinates $(x,y,\theta)$, where $\f=(x,y),$ $\b=\f- (\cos\theta, \sin\theta)$, with 
associated  global coordinate vector field  framing $\partial_x, \partial_y, \partial_\theta$.  (We take,  without loss of generality  the  bike length $\ell=1$. The general case reduces to  this case by an easy rescaling argument.)  The conjugate  fiber coordinates on $T^*Q$ are 
$p_x:=P_{\partial_x}$, $p_y:=P_{\partial_y}$,  $p_\theta:=P_{\partial_\theta}$.

 \begin{lemma} \label{lemma:contact}
 The no-skid condition defines on $Q$ a rank 2 distribution $D\subset TQ$, the kernel of the 1-form 
\be \eta:= d\theta-\ttc d y+\tts d x,  \ \mbox{ where } \ttc=\cos\theta, \ \tts=\sin\theta.
\ee 
It follows that $\eta\wedge d\eta=-dx\wedge dy\wedge d\theta$ is non-vanishing, hence $D$ is a contact distribution. 
\end{lemma}
\begin{proof}Let $\q(t)=(\b(t), \f(t))$ be a curve in $Q$ satisfying the non-skid condition. Let $\v:=\f-\b=(\ttc, \tts)$ and  decompose orthogonally $\dot\f=\dot\f_\|+\dot\f_\perp, $ where $\dot\f_\|,\dot\f_\perp$ are
the orthogonal projections of $\dot\f$ on $\v, \v^\perp$, respectively.   The condition $\|\f-\b\|=\mathrm{const}.$ and the no-skid condition $\dot\b\| \v$ are equivalent to $\dot\f_\|=\dot\b$. From $\f=\b+\v$ follows 
$\dot\f=\dot\b+\dot\v,$ hence $\dot\f_\|=\dot\b$ is equivalent to $\dot\v=\dot\f_\perp=
\dot\f-\dot\f_\|=\dot\f-\langle \dot\f, \v\rangle\v.$ In coordinates, this is $\dot\theta-\ttc\dot y+\tts\dot x=0.$ That is, $\dot\q\in \mathrm{Ker}(\eta).$
\end{proof}

Thus minimizing bike paths are arcs of normal geodesics. An orthonormal framing for $D=\mathrm{Ker}(\eta)$ is 
\be \label{eq:X12} X_1  :=  \partial_x -{\tts} \partial_\theta,  \
   X_2   :=\partial_{y} + {\ttc}\partial_{\theta},\ee 
with the associated 
\be\label{eq:Pi}
 P_1:=P_{X_1}=p_x -{\tts}p_\theta,\  P_2:=P_{X_2}=p_{y} + {\ttc}p_{\theta}.
 \ee 
 The Hamiltonian equations associated  to 
$ H=\left[(P_1)^2+(P_2)^2\right]/2$
are 
\begin{align}\label{eq:ham}
\begin{array}{lcl}
\dot x=\partial_{p_x}H=P_1=p_x -\tts p_\theta,&\quad& \dot p_x=-\partial_{x}H=0,\\
\dot y=\partial_{p_y}H=P_2= p_{y} + {\ttc}p_{\theta},&\quad&\dot p_y=-\partial_{y}H=0,\\
\dot\theta=\partial_{p_\theta}H=p_\theta+\ttc p_y-\tts p_x,&\quad&\dot p_\theta=-\partial_{\theta}H=p_\theta(\ttc p_x+\tts p_y).
\end{array}\end{align}

So $p_x, p_y$ are constant, as well as $H=(\dot x^2+\dot y^2)/2$. Fixing $H=1/2$  thus means that $\f(t)$ is parametrized by arc length.  Rotations act on the space of solutions  of equations \eqref{eq:ham} by  rotating $(x,y)$, $(p_x, p_y)$ and $(\ttc,\tts)$ simultaneously and shifting $\theta$, leaving $p_\theta$ unchanged. So we can assume  without loss of generality  say $p_y=0$ and $a:=p_x\geq 0.$ 
Equations \eqref{eq:ham} and $H=1/2$   now become 
\be\label{eq:ham1}\dot x=a -\tts p_\theta, \quad \dot y=\ttc p_{\theta},\quad 
 \dot\theta=p_\theta-a\tts ,  \quad\dot p_\theta=a \ttc p_\theta, \quad \dot\theta^2+a^2\ttc^2=1.
\ee

\begin{lemma}\label{lemma:kappa}
 Let $\kappa$ be the geodesic curvature of the front track $\f(t)=(x(t),y(t))$ of a solution to equations  \eqref{eq:ham1}.
  Then $\kappa=p_\theta.$
 \end{lemma}
 \begin{proof} 
We  calculate: $\kappa=\dot x\ddot y- \dot y\ddot x=\dot\theta +a \tts=p_\theta.$
\end{proof} 
We can thus rewrite  the unit speed geodesic equations  \eqref{eq:ham1} as
\be \label{eq:ham2}
\dot x=a -\tts \kappa, \quad \dot y=\ttc \kappa,\quad 
 \dot\theta=\kappa-a\tts ,  \quad\dot \kappa=a \ttc \kappa,\quad \dot\theta^2+a^2\ttc^2=1.
\ee

We are now ready to prove the 4 claims. 

\subsection{Proving the 1st claim (Theorem \ref{thm:main1})}

 \begin{proposition}\label{prop:kappaeq} The curvature $\kappa$ of the front track of a bicycle geodesic (solution to equations \eqref{eq:ham}), as a function of arc length $t$,  
 satisfies  the  `energy form' of the  elastica equation \eqref{eq:energy_form}, 
 $${\dot\kappa^2\over 2}+{\kappa^4\over 8}+{A\kappa^2\over 2}=B,$$ 
with 
\be\label{eq:AB}
A=-{a^2+1\over 2},\ B=-{(a^2-1)^2\over 8}.
\ee
That is, the front track  is a non-inflectional elastica or a straight line.
\end{proposition}
\begin{proof} The statement is invariant under rigid motions, so we can use instead equations \eqref{eq:ham2}. Then 
$\dot\kappa=a\ttc\kappa
$
and $\dot\theta^2=(\kappa-a\tts)^2=1-a^2\ttc^2,$ which simplifies to $2a\tts\kappa=\kappa^2+a^2-1.$ Thus 
$4\dot\kappa^2+(\kappa^2+a^2-1)^2=4a^2\kappa^2$, which gives the stated formula. 
\end{proof}

\begin{remark}
In the last paragraph of Section \ref{subsec:bike transport} below we sketch an alternative proof of Claim 1.  This alternative proof uses a relation between hyperbolic rolling geodesics
and bicycling geodesics and the fact that the hyperbolic rolling geodesics had been already  computed  and shown to correspond to elasticae \cite{Jurdjevic1}, \cite{Jurdjevic2}.

\end{remark}

\subsection{Proving the 2nd claim}  
The   `shape parameter'  of the front track is 
$\mu=-2B/A^2=(a^2-1)^2/(a^2+1)^2\in[0,1].$   Each $\mu\in(0,1)$ has 2 preimages, $a$ and $1/a$, one in $(0,1)$  the other in $ (1,\infty)$. It follows that  each NIE shape appears as a front track for two values of $a$. Let us determine the widths of these front tracks. Let  $\kappa_{max},  \kappa_{min}> 0$ be  the maximum and minimum value  of $\kappa$ along the front track. 

\begin{proposition}\label{prop:width}
The front track of a  solution to equations \eqref{eq:ham2} with $\kappa>0$ has $\kappa_{max}=1+a$ and $\kappa_{min}=|1-a|$. It follows that the width of the front track is
\begin{itemize}
\item 2 if $0< a\leq 1$  (a `wide' front track);
\item $2/a$ if  $a>1$    (a `narrow' front track). 
\end{itemize}
See Figure \ref{fig:geod}. 
\end{proposition}
\begin{proof} Since $\dot\kappa=0$ at  $\kappa_{max},  \kappa_{min}$, these critical values  must satisfy $(\kappa^2+a^2-1)^2=4a^2\kappa^2$ (see Proposition \ref{prop:kappaeq} and its proof). The solutions of this equation are $\kappa=\pm1\pm a.$  For $0<a<1$ the positive solutions are $1\pm a$, hence $\Delta\kappa=\kappa_{max}-\kappa_{min}=2a$. For 
$a>1$, the positive solutions are $a\pm 1$, hence  $\Delta\kappa=2$. From $\dot \kappa=a\dot y$ it  follows that $\Delta y=2$ in the 1st case and $2/a$ in the 2nd case, as needed. 
\end{proof}

\begin{remark}
One can also use equations  \eqref{eq:ham2} to find the widths of the respective back tracks: 
 $\left(1 - \sqrt{1 - a^2}\,\right)/a$ for  a `wide' front track,   and $ 2/a $ for a `narrow' front track (same as the width of the front track).
\end{remark}

\subsection{Proving the 3rd claim}

\begin{proposition}\label{prop:frame}
Consider a solution of equations  \eqref{eq:ham2} with $a>0$ and $\kappa>0$ (this can always be arranged for a non-linear front track  by appropriate reflections about the $x$ and $y$ axes). Then 
\begin{enumerate}
\item   $\theta=\pi/2$ at a point where $\kappa_{max}$ occurs. 
\item $\theta=-\pi/2$ at a point where $\kappa_{min}$ occurs  and $0<a<1$ (a `wide'  front  track). 
\item  $\theta=\pi/2$ at a point where $\kappa_{min}$ occurs  and  $a>1$ (a `narrow' front track).
 \end{enumerate}
See Figure \ref{fig:geod}. 
\end{proposition}

\begin{remark} The $a=1$ case is either a soliton, where $\kappa$ does not have a minimum, or a straight line. The  $a=0$ case is that of the unit circle and is safely left to the reader.
\end{remark}

\mn {\em Proof.}
By  equations \eqref{eq:ham2}, $\dot \kappa=a\ttc\kappa=a\dot y$. Thus in all 3 cases,  $\dot \kappa=0$ implies  $\ttc=\dot y=0$, which implies $\tts=\pm1$ and $\dot x=\pm 1$. We shall also use the formulas $\kappa_{max}=1+a$ and $\kappa_{min}=|1-a|$ from Proposition \ref{prop:width}, and $\dot x=a -\tts \kappa$  of equations  \eqref{eq:ham2}. 

\begin{enumerate}
\item Substitute  $\kappa=1+a$ in  $\dot x=a -\tts \kappa$  and get  $\dot x+\tts=a(1-\tts)$.  If $\tts=-1$ then $\dot x=2a+1>1$, which is impossible, hence $\tts=1, \dot x=-1$ and $\theta=\pi/2$.

\item If  $0<a<1$ then $\kappa_{min}=1-a.$ Substitute this in $\dot x=a -\tts \kappa$  and get  $\dot x+\tts=a(1+\tts)$.  If $\tts=1$ then $\dot x=2a-1.$  Together with $0<a<1$ this implies  $-1<\dot x<1$ which contradicts $\dot x=\pm 1.$ Hence  $\tts=-1$, $\dot x=1$ and $\theta=-\pi/2$ at a point where $\kappa_{min}$ occurs.

\item  If  $a>1$ then $\kappa_{min}=a-1.$ Substitute this in $\dot x=a -\tts \kappa$  and get  $\dot x-\tts=a(1-\tts)$.  If $\tts=-1$ then $\dot x=2a-1.$  Together with $a>1$ this implies  $\dot x>1$, which contradicts $\dot x=\pm 1.$ Hence  $\tts=1$, $\dot x=1$ and $\theta=\pi/2$ at a point where $\kappa_{min}$ occurs. \qed
\end{enumerate}

\subsection{Proving the  4th claim}

 Let $\gamma$ be a bicycle geodesic. A {\em vertex} of $\gamma$ is a point  on it  where an  extremum of the  curvature  of the front track  occurs ($\dot\kappa=0$). Our involution  $\Phi:Q\to Q$ is a \sR isometry,  hence  $\Phi\circ\gamma$ is a geodesic as well.

\begin{lemma} If $\gamma$ is a bicycle geodesic with $a\neq 1$ (that is, its front track is not a   line or soliton) then 
$\Phi$ maps  vertices of $\gamma$ to  vertices of $\Phi\circ\gamma.$
\end{lemma}
\begin{proof} Let $\gamma(t)=(x(t),y(t), \theta(t))$,  $\f(t)=(x(t), y(t))$  its front track and $\v(t)=(\cos \theta(t), \sin\theta(t))$ the frame direction. By Proposition \ref{prop:frame}, the vertices of $\gamma$ are the points where the frame is 
perpendicular to the front track, $\langle \dot\f,\v\rangle=0.$ Let $\tilde\gamma=\Phi\circ\gamma$. 
Then  $\tilde\f=\f-2\v$ and  $\tilde \v=-\v$, hence $\langle \dot{\tilde\f}, \tilde\v\rangle=
-\langle \dot{\f}-2\dot\v, \v\rangle=-\langle \dot{\f}, \v\rangle$, since  $\langle\v,\v\rangle=1$ implies $\langle\v,\dot\v\rangle=0.$ It follows that vertices of $\gamma$ and $\tilde\gamma$ occur simultaneously.
\end{proof}

The statement of Claim 4 is invariant under rigid motions and time reparametrizations, so we can assume, without loss of generality, that  $\gamma(t)=(x(t), y(t), \theta(t))$ satisfies equations \eqref{eq:ham2} with $a>0$, $a\neq 1$, $\kappa>0$ and  $\f_0=\f(0)$ is a point where $\kappa_{max}$ occurs. According to Proposition \ref{prop:frame} and its proof we then have $\theta_0=\pi/2$, 
$\v_0=(0,1)$,  $\kappa_0=1+a,$ $\dot\f_0=(-1,0)$ and $\ddot\f_0=(0, -\kappa_0)=-(0,1+a)$. 

Now let $\tilde\f(t)$ be the front track of $\tilde\gamma=\Phi\circ\gamma$. That is, $\tilde\f(t)=\f(t)-2\v(t).$

\begin{lemma}
{\em (1)}  $\dot{\tilde{\mathbf{\! f}}}_0=-\dot\f_0=(1,0),$ {\em (2) }
   $\ddot{\tilde{\mathbf{\! f}}}_0=(0,1-a). $
\end{lemma}
\begin{proof} (1) From equation \eqref{eq:ham2}, 
$\dot\theta=\kappa-a\tts.$ At $t=0$, $\kappa_0=1+a,$ $\theta_0=\pi/2$, hence $\dot\theta_0=1.$ Now $\dot\v=\dot\theta(-\tts,\ttc)$, hence $\dot\v_0=(-1,0).$ Thus $\dot{\tilde{\mathbf{\! f}}}_0=\dot\f_0-2\dot\v_0=(-1,0)-2(-1,0)=(1,0).$

\mn (2) From equations  \eqref{eq:ham2}, 
$\ddot\theta=\dot\kappa-a\dot\theta\ttc=a\ttc\kappa-a\ttc(\kappa-a\tts)= a\ttc\tts,$ hence $\ddot\theta_0=0.$
Thus $\ddot\v=\ddot\theta(-\tts,\ttc)- \dot\theta^2(\ttc,\tts)$ implies  $\ddot\v_0=(0,-1).$ It follows that 
 $\ddot{\tilde{\mathbf{\! f}}}_0=\ddot\f_0-2\ddot\v_0=(0,-1-a)-2(0,-1)= (0, 1+a).$
\end{proof}
\begin{figure}[ht]
\centering\def\svgwidth{\textwidth}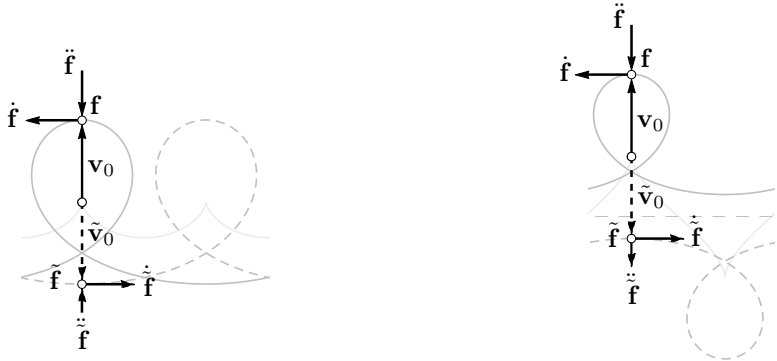 
\caption{The proof of claim 4 for `wide' (left) and
`narrow' (right) front tracks.}
\label{fig:claim4}
\end{figure}
We can conclude from the last lemma:
\begin{proposition}\label{prop:flip}
For any bicycle geodesic $\gamma$ with $a\neq 0, 1$, let $\tilde\gamma=\Phi\circ\gamma$. Then 
\begin{itemize}
\item if $0<a<1$ (wide front  track) then the front track of $\tilde\gamma$  is the result of translating the front track of $\gamma$  along its directrix for half its length;

\item 
 if  $a>1$ (narrow front  track) then the front track of $\tilde\gamma$ is the result of translating the front track of $\gamma$ along its directrix for half its length, then reflecting about the directrix. 
 \end{itemize}
 See Figure \ref{fig:claim4}. 
\end{proposition}

\section{The proof of  Theorem \ref{thm:main2}.}

We first prove half of Theorem  \ref{thm:main2}, the `if' part,  without invoking Theorem \ref{thm:main1}. Doing so   illustrates  some amusing  bicycling mathematics. 

By the `horizontal lift' of a front track $\f(t)$ we  mean any bicycle path  $({\bf b} (t), {\bf f}(t))$
whose front track projection is  the given curve ${\bf f} (t)$.  Since  a Euclidean line is a   metric line in $\R^2$, and since $\pi_f$ preserves lengths
when applied to bicycle paths,    every horizontal lift of a Euclidean   straight line  is a metric line in $Q$. 
However, just because the front wheel moves in a straight line    does not mean that  the  back wheel moves along the same straight line.   
Indeed, the back wheel typically   traces a  {\em  tractrix}  of width $\ell$ associated to the linear front track (unless at some moment the back wheel lies on the straight line, in which case
the back wheel also travels along the same straight line.)  See Figure \ref{fig:tractrix}.    

\begin{figure}[ht]
\centering
\centering\def\svgwidth{\textwidth}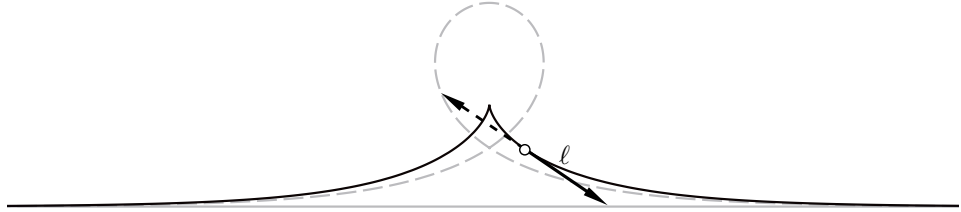
\caption{A tractrix: the back wheel track (dark solid curve) when the  front wheel travels along a straight line (light solid  horizontal line). The `flipped' front track is Euler's soliton (light dashed curve).}
\label{fig:tractrix}
\end{figure}

\begin{lemma}\label{lemma:tractrix}
Let $\gamma(t)=(\b(t), \f(t))\in Q$ be a horizontal lift of the straight line $\f(t)=(t,0)$. Then either $\b(t)=(t\pm\ell, 0)$ or $\b(t)$ is a tractrix of width $\ell$. Explicitly, 
$$\b(t)=\left(t-\ell \tanh \left[(t-t_0)/\ell\right],\ell\sech\left[(t-t_0)/\ell\right]\right),\quad  t_0\in\R.$$
The associated Euler soliton, obtained via the involution $\Phi$,  is 
$$\tilde\f(t)=2\b(t)-\f(t)=\left(t-2\ell \tanh \left[(t-t_0)/\ell\right],2\ell\sech\left[(t-t_0)/\ell\right]\right).
$$
\end{lemma}
\begin{proof}
From Lemma \ref{lemma:contact}, the horizontality condition on $\gamma$ is 
$\ell\dot\theta+\sin\theta=0$.  The general solution of this ODE  is $\theta(t)=-2 \cot ^{-1}\left(e^{(t-t_0)/\ell}\right)$, provided
that  $-\pi < \theta(0) < 0$,  from which the statement follows. 
(To get the solutions with  $0 < \theta(0) < \pi$ note that the equation is invariant under $\theta \to - \theta$.) 
\end{proof}

\begin{remark} Note that the tractrix   ${\bf b}(t)$ of Lemma  \ref{lemma:tractrix} tends to the earlier solutions  $(t \pm \ell, 0)$, as the `phase parameter'  $t_0 \to \pm \infty$.
\end{remark}

 We continue with the proof of the `if' part of Theorem \ref{thm:main2}. Let $c:\R\to\R^2$ be an  arc length parametrization of a straight line and  
 $\gamma: \R \to Q$  any horizontal lift of $c$.   As discussed immediately above, any horizontal  lift of a metric line must be a metric line,
hence  $\gamma$ is a metric line in $Q$.  The back track of such a lift is either a straight line or a tractrix of width $\ell$. In case the back track is a tractrix apply $\Phi$ to $\gamma$ and project back to   arrive at $\tilde c= \pi_f \circ \Phi \circ \gamma$.  
By the last lemma  any such  $\tilde c$ is an  Euler soliton of width  $2 \ell$. 
Since  isometries map metric lines to metric lines the curve  $\tilde\gamma=\Phi \circ \gamma$ is a metric line and so  
the Euler soliton  $\tilde c$ is the projection of a   metric line.  \qed

\mn 

Note that we can construct any Euler soliton of width $2 \ell$ in this way.  

\mn{\bf The rest of the proof of Theorem \ref{thm:main2} (the `only if' part).} We have just proven  that all infinite geodesics in $Q$ whose front tracks are straight lines or Euler solitons of width $2 \ell$ are metric lines.
To prove that there are no other metric lines in $Q$ we  invoke Theorem \ref{thm:main1}.
According to this theorem it  suffices to eliminate all the  non-inflectional  elasticae, other then the Euler soliton, as front tracks of metric lines in $Q$. 
Our proof follows the idea suggested by  Figure \ref{fig:conj}. Given any non-inflectional elasticae $\f(t)$, other then the Euler soliton, where $t$ is arc-length, we can rigidly rotate it so that its directrix is horizontal, that is, 
 $\f(t) = (x(t), y(t))$, where  $y(t)$ is periodic of some period $T > 0$ and  
 $x(t+T) = x(t)+L$ for some $L>0$.
 By translations in $x,y$ and $t$ we can further  assume that $\f(0)=(0,0)$ is a vertex of maximum curvature, so that  $x(0) = y(0) = 0$ and $\theta(0)=\pi/2$. (There are explicit expressions for $x(t)$ and $y(t)$ in terms of elliptic functions
 but we will not need these.)  
 
 Thus, after one period we have $\f(T) = (L,0)$.  Because $t$ is arc-length, the length of the elastica segment  between $\f(0)$ and $\f(T)$ is $T$.
 But  a straight horizontal line segment is the shortest curve connecting $\f(0)$ to $\f(T)$ and its length is $L$, so we must have that   $L<T$.

 \begin{figure}[ht]
\centering\def\svgwidth{\textwidth}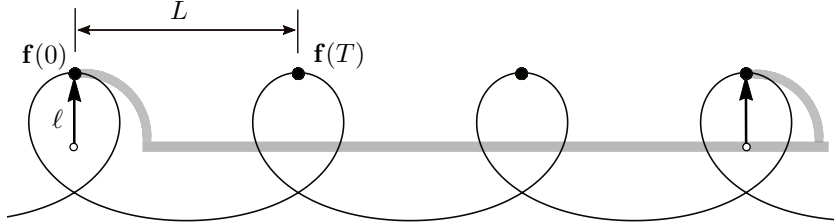
\caption{A shortcut.}
\label{fig:shortcut}
\end{figure}
 
 After $N$ periods the length of the elastica segment between $\f(0)$ and $\f(NT)$  is $NT$. As shown in Figure \ref{fig:shortcut}, we can find a shorter bike path between $\gamma(0)$ and  $\gamma(NT)$ for $N$ large enough, as follows: ride along a quarter circle of radius $\ell$ clockwise without moving the  back wheel; then ride along a straight line eastwards a distance of $NL$, then a quarter turn counterclockwise. The total length of this path is  $\pi\ell+NL$. For $N>\pi \ell/(T-L)$ this is shorter that $NT$. \qed

 \begin{remark}

Assuming Theorem \ref{thm:main1},   Theorem \ref{thm:main2} follows from the results described in  \cite{Sachkov3} where the  \sR geodesics for a \sR metric  on $\SEt$   isometric to our metric on $Q$
were   studied and   characterized as   solutions $\theta(t)$ to a family of pendulum equations.  Back in  our problem, that angle  $\theta$ is the angle  the bike frame makes with the $x$-axis.    
 In  \cite{Sachkov3} it was  proved that    a geodesic minimizes for all time  if and only if  $\theta(t) \equiv \mathrm{const}$ or $\theta(t)$ is a  non-periodic homoclinic solution of the pendulum problem. 
These conditions mean   that the front track of the bike moves along a straight line or is    an Euler soliton. 
\end{remark}

 \section{Loose Ends and Scattered Wheels}

 \subsection{Bicycling Correspondence}
 In the first part the proof  of Theorem \ref{thm:main2} (the ``if'' part) we build the Euler soliton out of  a Euclidean line by using the tractrix back-wheel curve
 as an intermediary step.   In the language of \cite{biking}, the  line and the Euler soliton are   in ``bicycle correspondence'' with each other,    the   tractrix  mediating
 the correspondence.    Take any sufficiently smooth  front wheel curve ${\bf f}(t)$.  Choose any one of its horizontal lifts $\gamma (t) = ({\bf b} (t),  {\bf f} (t))$.
 There are a circle's worth of such lifts, corresponding to an initial choice of point  ${\bf b} (t_0)$ on the circle of radius $\ell$ about ${\bf f} (t_0)$. 
 Apply the `flip' isometry $\Phi$ of Lemma \ref{flip}  to $\gamma$. Project $\Phi \circ \gamma$   back to the plane to arrive at  the  new front wheel curve $\tilde\f(t)=2\b(t) - \f(t),$
 which shares   its back wheel track ${\bf b} (t)$ with $\gamma(t)$.  
 Then the two  front wheel curves ${\bf f} (t)$ and $\tilde\f (t)$  are said to be in  bicycle correspondence. There are thus a circle's worth of bicycle correspondents to 
 ${\bf f} (t)$, corresponding to the circle's worth of choices for ${\bf b} (t)$.  
 
 \mn {\bf Question. }  Is a  bicycle correspondent to a projected geodesic always  a  projected geodesic?
 
 \mn {\bf No.}  The circle of radius $\ell$ is the projection of a geodesic corresponding to a back track fixed at this circle's center.
Most   bicycle correspondents of the  circle 
 are not elastica  and hence not projections of \sR geodesics.  
 It is interesting to note that these correspondents to the circle  are, instead,  {\it pressurized elasticae} which means their curvature
 $\kappa$ satisfies the ODE $\ddot \kappa +\frac{1}{2} \kappa ^3+ A \kappa = C$
 with a nonzero constant $C$. See Figure \ref{fig:pressurized}.

 \begin{figure}[ht]
\centering
\includegraphics[width=\linewidth]{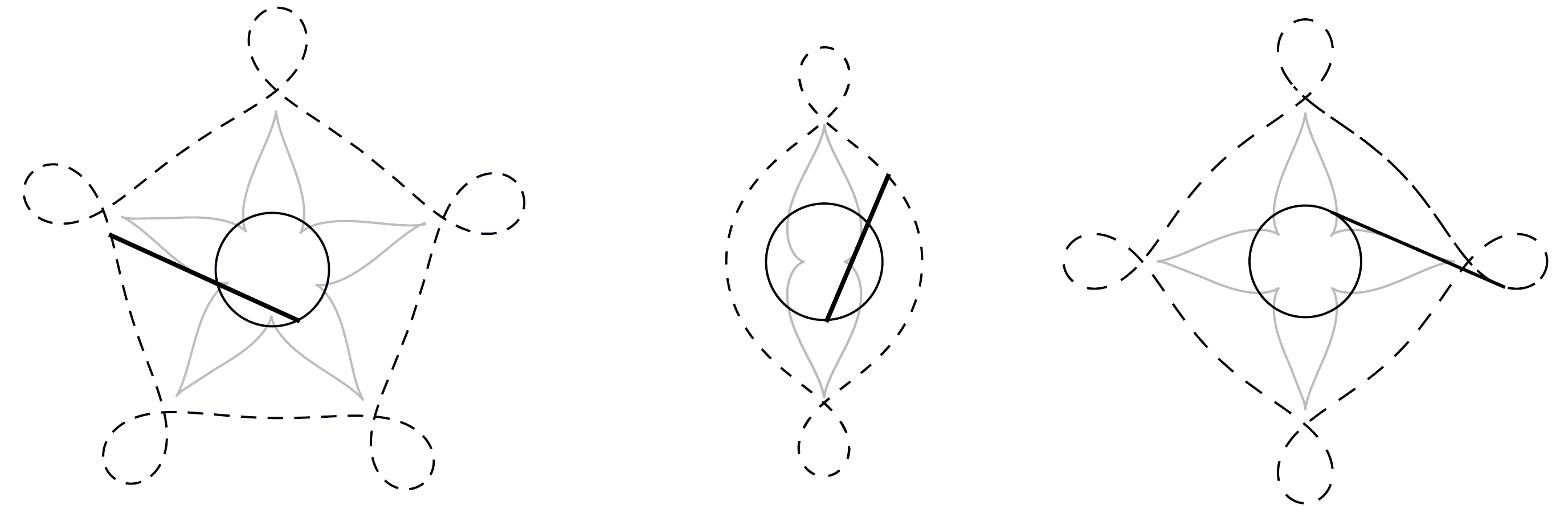} 
\caption{Pressurized elasticae (dashed curves), in bicycle correspondence with a circle, sharing  a common back track (light curve).}
\label{fig:pressurized}
\end{figure}
  
 \mn {\bf Question. } Is every  horizontal lift  of a projected geodesic  a geodesic?  
 
 \mn {\bf No.}    We just saw this above with the case of the circle.  Alternatively, see Claim 3 of Section \ref{sec:geod}.

  \subsection{Not of bundle type }

For the most familiar   \sR  submersions $M\to B$ the answer to the preceding question is {\bf yes:}  
every horizontal lift of every projected geodesic is a geodesic.   Examples include the  
 Heisenberg group, Carnot groups $G$ with $B = G/[G, G]  $, the Hopf fibration examples $S^3 \to S^2$ and the various
 principal bundle examples in \cite{tour}.     What makes these
 geometries different from bicycling geometry, group-theoretically speaking, is that for them the group of \sR isometries acts transitively on each fiber. 
  \begin{definition}  A \sR manifold $M$ is {\em of bundle type} if it admits a \sR submersion $\pi: M \to B$ and a Lie subgroup  $H \subset \Isom(M)$
 such that the fibers of $\pi$ are orbits of $H$. 
 \end{definition} 
 If $M \to B$ is of bundle type then,  necessarily,  every horizontal lift of a projected geodesic is a  geodesic. 
 So, our  bicycling \sR geometry with its front track projection $\pi_f:Q\to\R^2$ cannot  be  of bundle type. 

The  front track submersion  is a principal $S^1$-fibration, so that its   fibers are the orbits of a free $S^1$-action on $Q$, but this action cannot be  an action by isometries,
as we have just seen.  To see this fact directly,   fix a base point $\q_0\in Q$. Identify $\SEt\simeq Q$, $\q_0\mapsto g\q_0$. Then the induced \sR structure on $\SEt$ is invariant under {\em left}-translations by $\SEt$, while $\pi_f:\SEt\to \R^2$ is the quotient by  {\em right}-translations by  $S^1\subset\SEt$, the rotations about  $\pi_f(\q_0)\in \R^2$. 
 
  \begin{remark}In fact, this $S^1$-action  is  not even by  contact symmetries: right translation  $R_g$ by  an element 
 $g\in S^1$  defines a map of $\SEt$ which does not  preserve the contact distribution $D$. This failure  is easily seen by observing that $R_g$ acts on a bike path in $Q$  (a $D$-horizontal curve) by rotating the bike frame along the  path by a fixed angle, without changing the front track, producing `skidding' of the back wheel. 
 \end{remark}

\subsection{Other models for bicycling geometry } 
\label{sec:models}

The bicycling configuration space $Q$ can be identified, $\SEt$-equivariantly,  with  $ST \R^2$,  the space of unit tangent vectors to the plane.
Write elements of $ST \R^2$   as pairs $({\bf b}, {\bf v})$ where ${\bf b} \in \R^2$
and ${\bf v} \in \R^2$ is a unit vector attached at the point ${\bf b}$.  Identify ${\bf b}$ with the location of the back wheel
and ${\bf v}$ with the direction of the frame.  Then the isomorphism $ST \R^2 \to Q$ is  
\begin{equation}    ({\bf b}, {\bf v}) \mapsto ({\bf b}, {\bf f}),  \text { with } {\bf f}  = {\bf b} + \ell {\bf v}
\label{iso} 
\end{equation} 

The induced  contact distribution on $ST \R^2$, also denoted as $D$,  can be  described  by the condition that its smooth  integral curves $({\bf b} (t), {\bf v} (t))$ satisfy  $\dot\b(t)  \in \R\v(t) .  $
Write ${\bf b} = (x_b, y_b)$ and  ${\bf v} = (\cos \theta, \sin \theta),$ to define global coordinates  $(x_b, y_b, \theta)$ on $ST \R^2$.
In these coordinates a smooth    curve  $(x_b (t), y_b (t), \theta (t))$ is horizontal if and only if there is a smooth scalar function $\lambda (t)$ such that 
$\dot x_b = \lambda  \cos \theta ,
\dot y_b = \lambda  \sin\theta  $.     Eliminating $\lambda$, the contact distribution $D$ is given by the vanishing of the  contact form
\begin{equation}\label{eq:contact}
 (\sin\theta) d x_b - (\cos \theta) d y_b.
 \end{equation}

The   vector fields  
\begin{equation}
\bfS = \text{`straight ahead'}= (\cos\theta) \partial_{x_b}  + (\sin\theta)\partial_{y_b}
\label{straightahead} 
\end{equation}
and
\begin{equation}
\TT = \text{ ``turn''} =\partial_{\theta} 
\label{turn} 
\end{equation}
 clearly  frame  $D$.\footnote{The notation $\partial_\theta$  stands here  for a different  vector field from the one  in equations \eqref{eq:X12}, because we used  there  different coordinates, $(x,y,\theta)$. }
An  integral curve of  $\bfS$  corresponds to the bicycle moving along a straight line passing through the bike frame.   
An integral curve of $\TT$  corresponds to a  circus trick: the back wheel is stationary, 
marking  the center of a circle about which  the front wheel traces a circle of radius $\ell$.  To do this trick, the front wheel must be turned
at 90 degrees to the frame.    The front
wheel tracks of this line and a circle  are orthogonal.   By basic geometric considerations we see that 
$$ \langle \bfS, \bfS \rangle =1,  \quad \langle \bfS, \TT \rangle = 0, \quad \langle \TT, \TT \rangle = \ell^2.$$

\mn {\bf A second proof of Lemma \ref{flip}.}
Consider  the map  $\tilde \Phi (x_b, y_b, \theta) = (x_b, y_b, \theta + \pi)$ on $ST \R^2$.
One computes    $\tilde \Phi_* \bfS  = -\bfS,  \tilde \Phi_* \TT  = \TT$ which shows  that  $\tilde \Phi$  is a \sR {\em isometry}.  In
terms of our $({\bf b}, {\bf v})$   representation of $ST \R^2$ we have
 \begin{equation}
\label{flip2}
\tilde \Phi (({\bf b}, {\bf v})) =  ({\bf b}, - {\bf v}). 
\end{equation} 
Rewritten,  using   the isomorphism   (\ref{iso}), the  map (\ref{flip2})  becomes 
the map $ \Phi : Q \to Q$ of the lemma.\qed

\subsection{Bicycle parallel transport and hyperbolic geometry}
\label{subsec:bike transport}
Associated to a  \sR submersion $\pi: M \to B$ and a path   $c: I \to B$
we have a parallel transport map.  If the initial and final endpoints of $c$
are $f_0$ and $f_1$ then this is a map  $\Psi: \pi^{-1} (f_0) \to \pi^{-1} (f_1)$.  

\mn{\bf Question.} Is parallel transport for bicycling an isometry between fibers?

\mn{\bf No.}  One way to see this is via the following theorem. 
\begin{theorem} [Foote \cite{Foote}]
\label{foote}  The parallel transport map for bicycling  is a linear fractional transformation  of $S^1$.  Every  linear fractional transformation
can be obtained by parallel transport along  some closed  curve.    
\end{theorem}

 There is no metric on the circle for which the   the  group $\PSLt$ of linear fractional transformations acts by isometries,
so Foote's  theorem implies the {\bf  `no'} answer above. 
Again, if $M$ were of bundle type then  the answer to the above question would be {\bf yes:} parallel transport would be  an $H$-map
and hence an isometry.

Let us say a few words about what   parallel transport involves for   bicycling.  
Fix a front path $c$ joining two  front wheel locations ${\bf f}_0, {\bf f}_1$ in the plane.   The fiber $\pi_f ^{-1} ({\bf f}_0)$ is the circle of radius $\ell$ centered
at ${\bf f}_0$.  The points of this circle represent all ways of placing the back
wheel before   bicycling the front wheel  along the path $c$.     Choosing one such placement ${\bf b}_0 \in \pi_f ^{-1} ({\bf f}_0)$ leads to
a unique horizontal lift  $\gamma$ of $c$ starting at $\gamma(0) = ({\bf b}_0, {\bf f}_0)$. 
Here we assume that $c$ is parametrized by the unit interval $[0,1]$.   
Writing $\gamma(t) = ({\bf b}(t), c(t))$,  we have that the parallel
transport of  ${\bf b}_0$ along $c$ is 
${\bf b}(1)$, which is an element in the circle of radius $\ell$ about ${\bf f}_1$.    
So parallel transport, or holonomy, along $c$ is a diffeomorphism
between two circles, one centered at ${\bf f}_0$, the other at ${\bf f}_1$.   
 Use  translation and scaling  to identify each  circle    with the standard unit circle so that this  parallel transport becomes a map of the standard unit circle to itself.  
 Foote's Theorem above might be called  the `first theorem' of   `bicycling mathematics'.  See \cite{biking}.

Although the group  $\PSLt$ does not act isometrically on the circle,   it does act isometrically on
the hyperbolic plane. In fact $\PSLt$   equals  the group of rigid motions  of that plane.   
The configuration space for rolling a  hyperbolic plane on the Euclidean plane can be identified with 
$M:= \PSLt \times \R^2$ and   inherits, in a canonical way, a rank 2 \sR geometry such that the projection onto  $\R^2$
is a \sR submersion of bundle type. 
 \cite{Jurdjevic1, Jurdjevic2} prove that  its  geodesics project to  planar elasticae, both inflectional and non-inflectional.  

\subsection{A heuristic proof of Theorem 1.1.}  
\label{heuristics}
 Our  bicycle configuration space $Q$ can be identified with the circle bundle associated  to $M \to  \R^2$, where the structure group  $\PSLt$ 
 acts on the circle by fractional linear transformations, as per Foote's theorem.   
 Associated with any  plane curve $c:I \to \R^2$ we have its   hyperbolic rolling parallel transport,  an element  
 $k = k(c) \in \PSLt$ acting by left multiplication on the fibers. 
  The projection to $\R^2$  of a \sR  geodesic on   $M$   solves  the following {\it iso-holonomic  problem} (see \cite{Isohol} and   \cite{tour}, Chapter 11, especially  Theorem 11.8):
     among all plane curves $c$
 connecting given points  ${\bf f}_0$ to ${\bf f}_1$  and having a fixed hyperbolic transport $k = k(c) \in \PSLt$, find the shortest.     Now  
 imagine fixing the bicycle placement as well as the front wheel locations,
 which is to say, let us  fix  back  
   wheel locations  ${\bf b}_0, {\bf b_1}$,   writing  them as $\b_i = \f_i + \ell \v_i$.   Recall that  $k \in \PSLt$ acts on the unit circle. 
   Now, it may or may not be true that   $k ({\bf v_0}) = {\bf v_1}$.
   If not, let $k$ vary. Consider all $k \in \PSLt$ satisfying  $k ({\bf v}_0) =  {\bf v}_1$.
   For any such $k$ form the corresponding hyperbolic rolling geodesic $c$ for which $k(c) = k$.   Now, minimize
 the lengths of all such  $c$'s over all of  the $k$'s satisfying the condition that they take  ${\bf v}_0$ to ${\bf v}_1$. 
  The curve achieving this minimizer will
 be the   front wheel projection  of a   bike geodesic minimizing the length between $({\bf b}_0, {\bf f}_0)$ and $({\bf b}_1, {\bf f}_1)$, and will also be itself
 a particular type of hyperbolic rolling geodesic.   Since we know by  \cite{Jurdjevic1, Jurdjevic2} that hyperbolic rolling geodesics project to  
  elasticae, we're done!
 
 What makes this proof heuristic?  For one thing, the set of $k$'s over which we're minimizing  is a non-compact set, so we have no   guarantee that the  minimum exists.
 For another thing, the proof does not  single out the non-inflectional elasticae from all elasticae.

\subsection{Open Questions} The  bicycle correspondents  of a curve $c$ are the  result of the compositions  
$c \to \pi_f \circ  \Phi  \circ hc $, where $hc$ indicates  any of the circle's worth of horizontal
lifts of  the  front track $c$ and  where $\Phi$ is the  flipping isometry of Lemma \ref{flip}.  
  There are a number of hints in \cite{biking} that the  `transformation' of forming
bicycling correspondents  shares much in common with the B\"acklund transformations arising in the theory of integrable PDE.  

What is the family of curves that we get by forming the bicycle correspondents of elastica?    Repeat
and form  all the bicycle correspondents of  all the curves in this new family.  What do we get now?  Let us call this set the   `2nd generation' of correspondents
to elastica.  Keep going.  Does the procedure eventually close up,
or, do we get new curves at each generation?      Do the curves at the $n$-th generation  satisfy some ``nice'' ODE?
Are they  projections of   \sR geodesics on some \sR geometry  constructed  iteratively from $Q$ or  $\PSLt \times \R^2$? 

\mn 

We could  also ride our  bicycle   
on a sphere or    hyperbolic plane.  This change of bicycling arena corresponds to investigating a   left-invariant \sR structure on either  $\SOt$ or $\PSLt$,
these being the  unit tangent bundles  and also the group of rigid motions  of the sphere or hyperbolic plane, respectively.  (When bicycling on the sphere of radius $R$
one may need to insist that the frame's length is not equal to an integer multiple of $R\pi/2$
  to avoid various pathologies.  See \cite{Arnold2}  for interesting relations
   that might arise between front and back wheel curves when the spherical frame  length is  $R \pi/2$.) 
   How would our two  main theorems change?  Are the front wheel projections of  \sR   geodesics
still   elastica, meaning curves whose geodesic curvatures satisfy equation \ref{eq:energy_form}?  We guess so,  but   have not 
checked and are open to  surprises.  Would   bicycling on these non-Euclidean geometries   add to our understanding of how (or if)   these
different occurrences  of elastica in \sR geometry are related? Perhaps.

\begin{appendix}

\section{Proof of Theorem \ref{thm:isometries} on Isometries}

The group $\SEt$ of   orientation preserving isometries of the euclidean plane   acts freely and transitively by \sR isometries on the bicycling configuration space $Q$. Fixing a point $\q_0\in Q$,  we identify $\SEt\simeq Q$, $g\mapsto g\cdot \q_0$. This identification is $\SEt$-equivariant, hence induces a left-invariant \sR structure on $\SEt$,  given by its value at the identity $e\in \SEt$, a 2-dimensional subspace $D_e\subset\set,$ equipped with an inner product.

 To determine the isometry group $\Isom(\SEt)$ of this \sR structure we use two ingredients: (1) Cartan's equivalence method,  applied to the local classification  of 3-dimensional \sR manifolds of contact type; (2) A  calculation of $\Aut(\set,D_e)$, the group of automorphism of the Lie algebra of $\SEt$ preserving the contact plane at $e\in \SEt$ and the inner product. 

\mn (1) Let $M$ be a 3-dimensional \sR manifold of contact type (that is, $D\subset TM$ is bracket generating).  Similar to the  Riemannian case, one can use the Cartan method of equivalence to construct  a canonical connection on $TM$ and associated curvature tensor, whose vanishing is equivalent to  $M$ being  `flat', that is, locally isometric to the maximally  symmetric case, the \sR structure induced on $S^3$ from $S^2$ via the Hopf fibration $S^3\to S^2$, admitting  a 4-dimensional isometry group  (the standard  action of $\mathrm{U}_2$ on $\C^2\supset S^3$).  In the non-flat case, such as ours, the equivalence method shows that the isometry group, even the local one, is  at most 3-dimensional. It follows that  the space $\isom(M)$ of \sR Killing fields (vector fields whose flow acts by \sR isometries) is at most 3-dimensional.  A good reference for this circle of ideas is \cite{Keener}. 

\mn 

Now let $G$ be a   3-dimensional connected Lie group  with a left-invariant non-flat \sR structure  of contact type.  Let $\g=T_eG$ be its Lie algebra, equipped with the Lie bracket coming from the commutator of left-invariant vector fields. Let $L(G)\subset\Isom(G)$ be the (isomorphic)  image of the action of $G$ on itself by left translations. Then, by part (1)  above, $\dim[\Isom(G)]=3$. Let $\mathfrak{R}(G)$ be the 
right-invariant vector fields on $G$. They  generate left translations, hence $\mathfrak{R}(G)\subset\isom(G).$ But $\dim[\isom(G)]\leq 3,$ so $\mathfrak{R}(G)=\isom(G).$
Let  $\Isom_e(G)$ be the stabilizer of $e$ in $\Isom(G),$ a discrete subgroup. 

\begin{lemma}   $\Isom(G)=L(G)\rtimes \Isom_e(G)$. That is, $ L(G)$ is a normal subgroup of  $\Isom(G),$ $ L(G)\cap\Isom_e(G)=\{e\}$
and $\Isom(G)=L(G) \Isom_e(G)$. 

\end{lemma}

\begin{proof} By our assumptions on $G$ and dimensionality, $ L(G)$ is the identity component of  $\Isom(G)$ hence is a normal subgroup. If $L_g\in L(G)\cap\Isom_e(G)$ then $e=L_g(e)=ge=g$, hence $g=e$. Let $f\in \Isom(G)$ and  $g=f(e)$. Then $L_{g^-1}\circ f\in \Isom_e(G),$ hence $f\in L(G) \Isom_e(G).$
\end{proof}

\begin{lemma}  The map $\Isom_e(G)\to \GL(\g)$, $f\mapsto df_e$,  is {\em (a)} injective, {\em (b)} its image is contained in $\Aut(\g, D_e)$,  the group of Lie algebra automorphisms of $\g$ preserving $D_e$ and its inner product. \end{lemma}

\begin{proof} (a) An isometry of \sR connected manifolds of contact type is determined by its derivative at a single point (one can deduce it from the existence of a canonical Riemannian metric on such a manifold). Hence $f\mapsto df_e$  is injective. 

\sn (b) An isometry of a \sR manifold $M$ acts on its algebra of Killing vector fields $\isom(M)$ as an automorphism of  Lie algebras. In our case,  $\isom(G)=\mathfrak{R}(G)$ and  the evaluation map $\mathfrak{R}(G)\to \g$ is a Lie algebra anti-isomorphism, hence $df_e$ preserves the negative of the Lie bracket on $\g$, and thus the Lie bracket itself, that is, $df_e\in \Aut(\g)$. Since $f$ is a \sR isometry and fixes $e$, it leaves  $D_e$ invariant, acting on it by isometries. 
\end{proof}

\mn (2) After all these preliminaries, it remains to make some calculations in our case of $G=\SEt$, equipped with a left-invariant \sR structure induced by its action on the  bicycling configuration space $Q$. 

First, to show that such a  \sR structure  is non-flat, we  note that it is of contact type (see Lemma  \ref{lemma:contact}) and that an even stronger statement is known to hold; namely, that the CR structure associated to  such  a \sR structure on $\SEt$ (they are all equivalent) is not flat (the CR structure associated to a \sR structure  is obtained by keeping only the conformal structure on $D$, `forgetting scale'). See for example the calculation in \S7 of \cite{Bor2}. This statement was already known to \'E. Cartan, who classified all homogeneous 3-dimensional CR structures \cite{Cartan}.
We conclude that the group of \sR isometries $\Isom(\SEt)$ is 3-dimensional, where the identity component is generated by left translations of $\SEt$ on itself. Alternatively, one can use the last section of \cite{Hladky} to arrive at the same conclusion.

Next, we fix  a basis of $\set$,  given by the following Killing vector fields on $\R^2$,
$$\partial_x, \ \partial_y, \ \partial_\theta=x\partial_y-y\partial_x,
$$
satisfying 
$$[\partial_x, \partial_y]=0,\ [\partial_\theta,\partial_x]=-\partial_y,\ [\partial_\theta,\partial_y]=\partial_x.
$$

Next   fix  $\q_0=(\b_0,\f_0)\in Q$, where $\b_0=(-1,0),\f_0=(0,0)$ (we assume $\ell=1$, the general case follows easily from this case by a rescaling argument). In the coordinates  $(x,y,\theta)$ of Section \ref{sec:bge}, $\q_0$ is given by  $x_0=y_0=\theta_0=0$.  The actions of  
$\partial_x, \ \partial_y, \ \partial_\theta$ at $\b_0$ are $\partial_x, \partial_y, -\partial_y$, respectively and the no-skid condition at $\q_0$ is $\dot \b\| \partial_x$.  It follows that  $a\partial_x+b\partial_y+c\partial_\theta \in D_e\subset \set$ if and only if $b=c$. Thus $D_e$ is span by $X_1:=\partial_x, \ X_2:=\partial_y+\partial_\theta.$
They act at $\f_0=(0,0)$ by $\partial_x, \partial_y$, respectively, hence they form an \on basis for $D_e$. Let $X_3=[X_1,X_2]=\partial_y.$

\begin{lemma} $\Aut(\set, D_e)=\{id, \varphi_1, \varphi_2, \varphi_1\varphi_2\}=\{id, \varphi_1\}\cdot\{id, \varphi_2\}\simeq \Z_2\times\Z_2$, where $\varphi_1, \varphi_2$ are given in the basis $X_1, X_2, X_3$ by $\diag(1,-1,-1),\   \diag(-1,-1,1),$  respectively.
\end{lemma}

\begin{proof}One verifies easily  that $\varphi_1, \varphi_2\in  \Aut(\set, D_e)$ and that they generate  a  group $\{id, \varphi_1, \varphi_2, \varphi_1\varphi_2\}=\{id, \varphi_1\}\cdot\{id, \varphi_2\}\simeq \Z_2\times\Z_2$. It remains to show that any element $\varphi\in\Aut(\set, D_e)$ is in this group. The $X_1X_3$-plane (the linear  span  of $X_1, X_3$) is the only 
2-dimensional abelian ideal in $\set$, hence is $\varphi$-invariant. It follows that the  $X_1$-axis, the intersection of the $X_1X_3$-plane and $D_e$,  is $\varphi$-invariant. Being an isometry of  $D_e$,   $\varphi(X_1)=\varepsilon_1 X_1$, $\varphi(X_2)=\varepsilon_2 X_2$, with $\varepsilon_1,\varepsilon_2\in\{1,-1\}$. Being an automorphism, $\varphi(X_3)=\varphi([X_1, X_2])=[\varphi(X_1), \varphi(X_2)]=\varepsilon_1\varepsilon_2 X_3.$ \end{proof}

Next, we realize $\Aut(\set, D_e)$ by elements of $\Isom_e(\SEt)$. With each element $f\in\Isom(Q)$ is associated an element $\tilde f\in \Isom(\SEt)$ via the identification $\SEt\simeq Q$, $g\mapsto g\q_0.$ For $g\in\SEt\subset\Isom(Q),$ $\tilde g=L_g$. 
 Let $\rho\in\Et$ be reflection about the $x$ axis. Using complex notation, $\rho(\z)=\bar \z$. Let $g_{\u, \w}\in\SEt$, $\z\mapsto \u\z+\w,$ where $\u,\w\in\C$ and $|\u|=1$. 
Then $\rho g_{\u, \z_0}=g_{\bar \u, \bar\z_0}\rho$. Hence $\tilde\rho\cdot g_{\u, \w}=g_{\bar\u, \bar\w}.$ Similarly, $\tilde\Phi\cdot g_{\u, \w}=g_{-\u, \w-2\u}.$

\begin{lemma}$\varphi_1=d\tilde\rho_e,\ \varphi_2=d\tilde f_e,$ 
where $\rho$ is reflection about the $x$ axis,  $f=\rho'\Phi$, and $\rho'$ is the reflexion about the line $x=-1$. \end{lemma}
 \begin{proof}This is a routine verification. 1st  verify that both $\rho, f$ leave $\q_0$ fixed, so $\tilde\rho, \tilde f\in\Isom_{e}(\SEt)$. Next check that $d\tilde\rho_e:\partial_x\mapsto \partial_x, \ \partial_y\mapsto -\partial_y,\ \partial_\theta\mapsto-\partial_\theta.$ It follows that $d\tilde\rho_e=\varphi_1.$ Next check that 
 $d\tilde f_e:\partial_x\mapsto -\partial_x, \ \partial_y\mapsto -\partial_y,\ \partial_\theta\mapsto-2\partial_y-\partial_\theta.$ It follows that $d\tilde f_e=\varphi_2.$
 \end{proof}
 

 \begin{corollary}Let $\Gamma=\{id, \rho, \Phi, \rho\Phi\}\subset\Isom(Q),$ where $\rho\in \Et\setminus \SEt$ (a reflection about a line). Then $\Gamma=\{id, \rho\}\cdot \{id, \Phi\}\simeq\Z_2\times\Z_2$ and $\Isom(Q)=\SEt\rtimes \Gamma$.
 \end{corollary}
 
 \begin{proof}Clearly, $\SEt\cap \Gamma=\{id\}$, so it remains to show that $\SEt\cdot\Gamma=\Isom(Q)$. We can assume, by conjugating by an element of $\SEt$ that maps the fixed line of $\rho$ to the $x$-axis, that $\rho$ is the reflection about the $x$-axis (the same $\rho$ as in the last lemma).  By  the previous lemmas, $\Isom(Q)=\SEt\cdot\Gamma_0, $ where $\Gamma_0=\Isom_{\q_0}(Q)=\{id, \rho, \rho'\Phi, \rho\rho'\Phi \}.$ Now $\rho\Phi\equiv \rho'\Phi$ and 
 $\Phi\equiv \rho\rho'\Phi $ (mod $\SEt$), hence $\SEt\cdot\Gamma=\SEt\cdot\Gamma_0=\Isom(Q)$.\end{proof}

\end{appendix}

%

\end{document}

%% file: figures/bike.pdf_tex
\begingroup%
  \makeatletter%
  \providecommand\color[2][]{%
    \errmessage{(Inkscape) Color is used for the text in Inkscape, but the package 'color.sty' is not loaded}%
    \renewcommand\color[2][]{}%
  }%
  \providecommand\transparent[1]{%
    \errmessage{(Inkscape) Transparency is used (non-zero) for the text in Inkscape, but the package 'transparent.sty' is not loaded}%
    \renewcommand\transparent[1]{}%
  }%
  \providecommand\rotatebox[2]{#2}%
  \ifx\svgwidth\undefined%
    \setlength{\unitlength}{999.99993896bp}%
    \ifx\svgscale\undefined%
      \relax%
    \else%
      \setlength{\unitlength}{\unitlength * \real{\svgscale}}%
    \fi%
  \else%
    \setlength{\unitlength}{\svgwidth}%
  \fi%
  \global\let\svgwidth\undefined%
  \global\let\svgscale\undefined%
  \makeatother%
  \begin{picture}(1,0.15000001)%
    \put(0.52065169,0.06120489){\color[rgb]{0,0,0}\transparent{0.95999998}\makebox(0,0)[lb]{\smash{$\ell$}}}%
    \put(0.6360886,0.0169412){\color[rgb]{0,0,0}\transparent{0.95999998}\makebox(0,0)[lb]{\smash{$\f$}}}%
    \put(0.3578881,0.11454212){\color[rgb]{0,0,0}\transparent{0.95999998}\makebox(0,0)[lb]{\smash{$\b$}}}%
    \put(0,0){\includegraphics[width=\unitlength,page=1]{bike.pdf}}%
  \end{picture}%
\endgroup%

%% file: figures/elasticae.pdf_tex
\begingroup%
  \makeatletter%
  \providecommand\color[2][]{%
    \errmessage{(Inkscape) Color is used for the text in Inkscape, but the package 'color.sty' is not loaded}%
    \renewcommand\color[2][]{}%
  }%
  \providecommand\transparent[1]{%
    \errmessage{(Inkscape) Transparency is used (non-zero) for the text in Inkscape, but the package 'transparent.sty' is not loaded}%
    \renewcommand\transparent[1]{}%
  }%
  \providecommand\rotatebox[2]{#2}%
  \ifx\svgwidth\undefined%
    \setlength{\unitlength}{924.03121948bp}%
    \ifx\svgscale\undefined%
      \relax%
    \else%
      \setlength{\unitlength}{\unitlength * \real{\svgscale}}%
    \fi%
  \else%
    \setlength{\unitlength}{\svgwidth}%
  \fi%
  \global\let\svgwidth\undefined%
  \global\let\svgscale\undefined%
  \makeatother%
  \begin{picture}(1,0.32466436)%
    \put(0,0){\includegraphics[width=\unitlength,page=1]{elasticae.pdf}}%
    \put(0.53832155,0.00684031){\color[rgb]{0,0,0}\makebox(0,0)[lb]{\smash{\sm soliton}}}%
    \put(0.21600751,0.02838706){\color[rgb]{0,0,0}\makebox(0,0)[lb]{\smash{\sm Inflectional }}}%
    \put(0.69947854,0.03114188){\color[rgb]{0,0,0}\makebox(0,0)[lb]{\smash{\sm Non-inflectional }}}%
    \put(0.53832155,0.0360601){\color[rgb]{0,0,0}\makebox(0,0)[lb]{\smash{\sm Euler's}}}%
  \end{picture}%
\endgroup%

%% file: figures/parameters.pdf_tex
\begingroup%
  \makeatletter%
  \providecommand\color[2][]{%
    \errmessage{(Inkscape) Color is used for the text in Inkscape, but the package 'color.sty' is not loaded}%
    \renewcommand\color[2][]{}%
  }%
  \providecommand\transparent[1]{%
    \errmessage{(Inkscape) Transparency is used (non-zero) for the text in Inkscape, but the package 'transparent.sty' is not loaded}%
    \renewcommand\transparent[1]{}%
  }%
  \providecommand\rotatebox[2]{#2}%
  \ifx\svgwidth\undefined%
    \setlength{\unitlength}{825bp}%
    \ifx\svgscale\undefined%
      \relax%
    \else%
      \setlength{\unitlength}{\unitlength * \real{\svgscale}}%
    \fi%
  \else%
    \setlength{\unitlength}{\svgwidth}%
  \fi%
  \global\let\svgwidth\undefined%
  \global\let\svgscale\undefined%
  \makeatother%
  \begin{picture}(1,0.40909091)%
    \put(0,0){\includegraphics[width=\unitlength,page=1]{parameters.pdf}}%
    \put(0.16619225,0.02408447){\color[rgb]{0,0,0}\makebox(0,0)[lb]{\smash{\sm Circles}}}%
    \put(0.16430623,0.24551409){\color[rgb]{0,0,0}\makebox(0,0)[lb]{\smash{\sm Solitons}}}%
    \put(0,0){\includegraphics[width=\unitlength,page=2]{parameters.pdf}}%
    \put(0.90376588,0.24936828){\color[rgb]{0,0,0}\makebox(0,0)[lb]{\smash{$A$}}}%
    \put(0.63125346,0.1658487){\color[rgb]{0,0,0}\makebox(0,0)[lb]{\smash{}}}%
    \put(0.72194771,0.39482284){\color[rgb]{0,0,0}\makebox(0,0)[lb]{\smash{$B$}}}%
    \put(0.07612441,0.02402032){\color[rgb]{0,0,0}\makebox(0,0)[lb]{\smash{$\mu=1$}}}%
    \put(0.07489091,0.24486157){\color[rgb]{0,0,0}\makebox(0,0)[lb]{\smash{$\mu=0$}}}%
    \put(0,0){\includegraphics[width=\unitlength,page=3]{parameters.pdf}}%
    \put(0.07612441,0.3362912){\color[rgb]{0,0,0}\makebox(0,0)[lb]{\smash{$\mu=-2B/ A^2$}}}%
    \put(0,0){\includegraphics[width=\unitlength,page=4]{parameters.pdf}}%
    \put(0.59157892,0.2862912){\color[rgb]{0,0,0}\makebox(0,0)[lb]{\smash{\sm Inflectional elasticae}}}%
    \put(0.46584839,0.04909837){\color[rgb]{0,0,0}\rotatebox{90}{\makebox(0,0)[lb]{\smash{\sm  Non-inflectional}}}}%
    \put(0.49312112,0.06909837){\color[rgb]{0,0,0}\rotatebox{90}{\makebox(0,0)[lb]{\smash{\sm elasticae}}}}%
  \end{picture}%
\endgroup%

%% file: figures/flip.pdf_tex
\begingroup%
  \makeatletter%
  \providecommand\color[2][]{%
    \errmessage{(Inkscape) Color is used for the text in Inkscape, but the package 'color.sty' is not loaded}%
    \renewcommand\color[2][]{}%
  }%
  \providecommand\transparent[1]{%
    \errmessage{(Inkscape) Transparency is used (non-zero) for the text in Inkscape, but the package 'transparent.sty' is not loaded}%
    \renewcommand\transparent[1]{}%
  }%
  \providecommand\rotatebox[2]{#2}%
  \ifx\svgwidth\undefined%
    \setlength{\unitlength}{600bp}%
    \ifx\svgscale\undefined%
      \relax%
    \else%
      \setlength{\unitlength}{\unitlength * \real{\svgscale}}%
    \fi%
  \else%
    \setlength{\unitlength}{\svgwidth}%
  \fi%
  \global\let\svgwidth\undefined%
  \global\let\svgscale\undefined%
  \makeatother%
  \begin{picture}(1,0.125)%
    \put(0,0){\includegraphics[width=\unitlength,page=1]{flip.pdf}}%
    \put(0.66036925,0.09067345){\color[rgb]{0,0,0}\makebox(0,0)[lb]{\smash{$\tilde\f=2\b-\f$}}}%
    \put(0.4789251,0.02174594){\color[rgb]{0,0,0}\makebox(0,0)[lb]{\smash{$\b$}}}%
    \put(0.29951428,0.0209236){\color[rgb]{0,0,0}\makebox(0,0)[lb]{\smash{$\f$}}}%
    \put(0.54809851,0.04419286){\color[rgb]{0,0,0}\makebox(0,0)[lb]{\smash{$\tilde\v$}}}%
    \put(0.40426533,0.0163251){\color[rgb]{0,0,0}\makebox(0,0)[lb]{\smash{$\v$}}}%
    \put(0.17830141,0.39314518){\color[rgb]{0,0,0}\makebox(0,0)[lb]{\smash{}}}%
    \put(0.25148802,0.66335812){\color[rgb]{0,0,0}\makebox(0,0)[lb]{\smash{}}}%
    \put(0.05434783,0.1576087){\color[rgb]{0,0,0}\makebox(0,0)[lt]{\begin{minipage}{0.57744565\unitlength}\raggedright \end{minipage}}}%
  \end{picture}%
\endgroup%

%% file: figures/claim4.pdf_tex
\begingroup%
  \makeatletter%
  \providecommand\color[2][]{%
    \errmessage{(Inkscape) Color is used for the text in Inkscape, but the package 'color.sty' is not loaded}%
    \renewcommand\color[2][]{}%
  }%
  \providecommand\transparent[1]{%
    \errmessage{(Inkscape) Transparency is used (non-zero) for the text in Inkscape, but the package 'transparent.sty' is not loaded}%
    \renewcommand\transparent[1]{}%
  }%
  \providecommand\rotatebox[2]{#2}%
  \ifx\svgwidth\undefined%
    \setlength{\unitlength}{1125bp}%
    \ifx\svgscale\undefined%
      \relax%
    \else%
      \setlength{\unitlength}{\unitlength * \real{\svgscale}}%
    \fi%
  \else%
    \setlength{\unitlength}{\svgwidth}%
  \fi%
  \global\let\svgwidth\undefined%
  \global\let\svgscale\undefined%
  \makeatother%
  \begin{picture}(1,0.36666667)%
    \put(0,0){\includegraphics[width=\unitlength,page=1]{claim4.pdf}}%
    \put(0.13471584,0.07986864){\color[rgb]{0,0,0}\makebox(0,0)[lb]{\smash{$\tilde\f$}}}%
    \put(0.17723587,0.26326091){\color[rgb]{0,0,0}\makebox(0,0)[lb]{\smash{$\f$}}}%
    \put(0.14994016,0.30510904){\color[rgb]{0,0,0}\makebox(0,0)[lb]{\smash{$\ddot\f$}}}%
    \put(0.09046661,0.24942448){\color[rgb]{0,0,0}\makebox(0,0)[lb]{\smash{$\dot\f$}}}%
    \put(0.17471571,0.20284558){\color[rgb]{0,0,0}\makebox(0,0)[lb]{\smash{$\v_0$}}}%
    \put(0.16471171,0.01652832){\color[rgb]{0,0,0}\makebox(0,0)[lb]{\smash{$\ddot{\tilde\f}$}}}%
    \put(0.23376155,0.07186602){\color[rgb]{0,0,0}\makebox(0,0)[lb]{\smash{$\dot{\tilde\f}$}}}%
    \put(0.17469421,0.1303725){\color[rgb]{0,0,0}\makebox(0,0)[lb]{\smash{$\tilde{\v}_0$}}}%
    \put(0.09509409,0.20967743){\color[rgb]{0,0,0}\makebox(0,0)[lb]{\smash{}}}%
    \put(0,0){\includegraphics[width=\unitlength,page=2]{claim4.pdf}}%
    \put(0.7213825,0.12393321){\color[rgb]{0,0,0}\makebox(0,0)[lb]{\smash{$\tilde\f$}}}%
    \put(0.75592415,0.31325553){\color[rgb]{0,0,0}\makebox(0,0)[lb]{\smash{$\f$}}}%
    \put(0.72727352,0.35444237){\color[rgb]{0,0,0}\makebox(0,0)[lb]{\smash{$\ddot\f$}}}%
    \put(0.67079117,0.29836711){\color[rgb]{0,0,0}\makebox(0,0)[lb]{\smash{$\dot\f$}}}%
    \put(0.75204914,0.24951226){\color[rgb]{0,0,0}\makebox(0,0)[lb]{\smash{$\v_0$}}}%
    \put(0.74249815,0.0638903){\color[rgb]{0,0,0}\makebox(0,0)[lb]{\smash{$\ddot{\tilde\f}$}}}%
    \put(0.80901406,0.12448427){\color[rgb]{0,0,0}\makebox(0,0)[lb]{\smash{$\dot{\tilde\f}$}}}%
    \put(0.75271036,0.16906071){\color[rgb]{0,0,0}\makebox(0,0)[lb]{\smash{$\tilde{\v}_0$}}}%
    \put(0.13412694,0.353791){\color[rgb]{0,0,0}\makebox(0,0)[lb]{\smash{}}}%
  \end{picture}%
\endgroup%

%% file: figures/tractrix.pdf_tex
\begingroup%
  \makeatletter%
  \providecommand\color[2][]{%
    \errmessage{(Inkscape) Color is used for the text in Inkscape, but the package 'color.sty' is not loaded}%
    \renewcommand\color[2][]{}%
  }%
  \providecommand\transparent[1]{%
    \errmessage{(Inkscape) Transparency is used (non-zero) for the text in Inkscape, but the package 'transparent.sty' is not loaded}%
    \renewcommand\transparent[1]{}%
  }%
  \providecommand\rotatebox[2]{#2}%
  \ifx\svgwidth\undefined%
    \setlength{\unitlength}{1125bp}%
    \ifx\svgscale\undefined%
      \relax%
    \else%
      \setlength{\unitlength}{\unitlength * \real{\svgscale}}%
    \fi%
  \else%
    \setlength{\unitlength}{\svgwidth}%
  \fi%
  \global\let\svgwidth\undefined%
  \global\let\svgscale\undefined%
  \makeatother%
  \begin{picture}(1,0.23433333)%
    \put(0,0){\includegraphics[width=\unitlength,page=1]{tractrix.pdf}}%
    \put(0.57998496,0.04828453){\color[rgb]{0,0,0}\makebox(0,0)[lb]{\smash{$\ell$}}}%
    \put(0,0){\includegraphics[width=\unitlength,page=2]{tractrix.pdf}}%
  \end{picture}%
\endgroup%

%% file: figures/shortcut.pdf_tex
\begingroup%
  \makeatletter%
  \providecommand\color[2][]{%
    \errmessage{(Inkscape) Color is used for the text in Inkscape, but the package 'color.sty' is not loaded}%
    \renewcommand\color[2][]{}%
  }%
  \providecommand\transparent[1]{%
    \errmessage{(Inkscape) Transparency is used (non-zero) for the text in Inkscape, but the package 'transparent.sty' is not loaded}%
    \renewcommand\transparent[1]{}%
  }%
  \providecommand\rotatebox[2]{#2}%
  \ifx\svgwidth\undefined%
    \setlength{\unitlength}{1125bp}%
    \ifx\svgscale\undefined%
      \relax%
    \else%
      \setlength{\unitlength}{\unitlength * \real{\svgscale}}%
    \fi%
  \else%
    \setlength{\unitlength}{\svgwidth}%
  \fi%
  \global\let\svgwidth\undefined%
  \global\let\svgscale\undefined%
  \makeatother%
  \begin{picture}(1,0.26666667)%
    \put(0,0){\includegraphics[width=\unitlength,page=1]{shortcut.pdf}}%
    \put(0.24275889,0.23566853){\color[rgb]{0,0,0}\makebox(0,0)[lb]{\smash{$L$}}}%
    \put(0.08697537,0.1898205){\color[rgb]{0,0,0}\makebox(0,0)[lb]{\smash{$\f(0)$}}}%
    \put(0.39327508,0.1911829){\color[rgb]{0,0,0}\makebox(0,0)[lb]{\smash{$\f(T)$}}}%
    \put(0.11874539,0.12126575){\color[rgb]{0,0,0}\makebox(0,0)[lb]{\smash{$\ell$}}}%
    \put(0,0){\includegraphics[width=\unitlength,page=2]{shortcut.pdf}}%
  \end{picture}%
\endgroup%